\documentclass[12pt]{article}
\usepackage{latexsym}
\usepackage{amssymb,amsmath}
\usepackage{color}
    \definecolor{gray}{gray}{0.5}

\begin{document}
\baselineskip=.25in\parindent=30pt

\newtheorem{theorem}{Theorem}[section]
\newtheorem{prop}[theorem]{Proposition}
\newtheorem{lemma}[theorem]{Lemma}
\newtheorem{cor}[theorem]{Corollary}
\newtheorem{dfn}[theorem]{Definition}
\newtheorem{remark}[theorem]{Remark}
\newtheorem{example}[theorem]{Example}
\newtheorem{assumption}[theorem]{Assumption}

\def\bs{\vskip24pt}
\def\ms{\vskip12pt}
\def\ss{\vskip6pt}

\def\cl{\centerline}
\def\nt{\noindent}

\newcommand{\fn}{\footnote}
\newcommand{\co}{c^{\,\text{com}}}
\newcommand{\Lcom}{\Lambda}
\newcommand\col{\operatorname{col}}
\newcommand\sgn{\operatorname{sgn}}
\newcommand{\Out}{\operatorname{Out}}
\newcommand\eps{\varepsilon}
\newcommand\supp{\operatorname{supp}}
\newcommand\diag{\operatorname{diag}}
\newcommand\rank{\operatorname{rank}}
\newcommand\gt{\theta}
\newcommand\gvf{\varphi}
\newcommand\cg{{\mathcal G}}
\newcommand\gve{\varepsilon}
\newcommand\go{\omega}
\newcommand\gO{\Omega}
\newcommand\tr{\operatorname{tr}}
\newcommand\gs{\sigma}
\newcommand\ga{\alpha}
\newcommand\gd{\delta}
\newcommand\gD{\Delta}
\newcommand\D{{\mathbb D}}
\newcommand\Cyc{\operatorname{Cyc}}
\renewcommand\mod{\operatorname{mod}}
\newcommand\Llr{\Longleftrightarrow}
\newcommand\esssup{\operatorname{esssup}}
\newcommand\essinf{\operatorname{essinf}}
\newcommand\gS{\Sigma}
\newcommand{\gL}{\Lambda}
\newcommand{\gl}{\lambda}
\newcommand\cd{{\mathcal D}}
\newcommand\C{{\mathbb C}}
\newcommand\R{{\mathbb R}}
\newcommand\N{{\mathbb N}}
\newcommand\wt{\widetilde}
\newcommand{\ca}{{\mathfrak A}}
\newcommand{\cb}{{\mathcal B}}
\newcommand{\cD}{{\mathcal D}}
\newcommand{\ce}{{\mathfrak E}}
\newcommand{\cs}{{\mathcal S}}
\newcommand{\ch}{{\mathcal H}}
\newcommand{\ck}{{\mathcal K}}
\newcommand{\cm}{{\mathcal M}}
\newcommand{\cy}{{\mathfrak Y}}
\newcommand{\cx}{{\mathcal X}}
\newcommand{\cc}{{\mathfrak C}}
\newcommand{\cf}{{\mathcal F}}
\newcommand\cv{{\mathfrak V}}
\newcommand\bz{{\bf z}}
\newcommand\by{{\bf y}}
\newcommand{\bk}{{\bf K}}
\newcommand{\EQPR}{\operatorname{\rm Eqpr}}
\newcommand{\bq}{{\bf q}}
\newcommand{\grad}{\operatorname{grad}}
\newcommand{\vol}{\operatorname{vol}}
\newcommand{\br}{{\bf r}}
\newcommand{\cP}{{\mathcal P}}
\newcommand{\cR}{{\mathcal R}}
\newcommand{\bd}{{\bf D}}
\newcommand{\bx}{{\bf x}}
\newcommand{\bX}{{\bf X}}
\newcommand{\bY}{{\bf Y}}
\newcommand{\bb}{{\bf b}}
\newcommand{\bc}{{\bf c}}
\newcommand{\be}{{\bf e}}
\newcommand{\bm}{{\bf m}}
\newcommand{\bN}{{\bf N}}
\newcommand{\bA}{{\bf A}}
\newcommand{\Jac}{{\operatorname{Jac}}}
\newcommand{\ERC}{{\operatorname{ERC}}}
\newcommand{\REC}{{\operatorname{REC}}}
\newcommand{\Var}{{\operatorname{\rm Var}}}
\newcommand{\Cov}{{\operatorname{\rm Cov}}}
\newcommand{\Corr}{{\operatorname{\rm Corr}}}
\newcommand{\const}{\operatorname{\rm const}}
\renewcommand{\span}{\operatorname{\rm span}}
\newcommand{\ba}{{\bf a}}
\newcommand{\bL}{{\bf L}}
\newcommand{\bP}{{\bf P}}
\newcommand{\bG}{{\bf G}}
\newcommand{\bM}{{\bf M}}
\newcommand{\bT}{{\bf T}}
\newcommand{\bW}{{\bf W}}
\newcommand{\bZ}{{\bf Z}}
\newcommand{\bw}{{\bf w}}
\newcommand{\bfc}{{\bf c}}
\newcommand{\Z}{{\mathbb Z}}
\newcommand{\ones}{{\bf 1}}
\newcommand{\T}{{\mathbb T}}
\newcommand{\wM}{{\widetilde M}}
\newcommand{\pF}{{\rm P}_{\hskip-2pt\cf}}
\newcommand{\bpF}{{\bf P}_{\hskip-2pt\cf}}
\newcommand{\pG}{{\rm P}_{\hskip-1pt\cg}}
\newcommand{\bpG}{{\bf P}_{\hskip-2pt\cg}}
\newcommand{\pH}{{\rm P}_{\hskip-2pt\ch}}
\newcommand{\bpH}{{\bf P}_{\hskip-2pt\ch}}
\newcommand{\pQ}{{\rm Q}}
\newcommand{\bpQ}{{\bf Q}}
\newcommand{\Dis}{\operatorname{Distr}}
\newcommand{\Prob}{\operatorname{Prob}}
\newcommand{\sce}{{\scriptstyle \ce}}
\newcommand{\scc}{{\scriptstyle \cc}}
\newcommand{\bsc}{{\scriptstyle \bf C}}
\newcommand{\scy}{{\scriptstyle \cy}}
\newcommand{\scv}{{\scriptstyle \cv}}
\newcommand{\scrv}{{\scriptstyle V}}
\newcommand{\sct}{{\scriptstyle T}}
\newcommand{\ru}{{\rm u}}
\newcommand{\rp}{{\rm p}}
\newcommand\com{{\rm cm}}
\newcommand\comt{{\rm cm}_{\,\hskip.3pt t}}
\newcommand{\bcm}{{\bf cm}}
\newcommand{\cM}{\mathcal M}

\def\qed{\quad\vrule height4pt width4pt depth0pt}
\def\argmax{\operatornamewithlimits{arg\,max \quad}}

\def\title #1{\begin{center}
{\Large {\bf #1}}
\end{center}}
\def\author #1{\begin{center} {\large #1}
\end{center}}

\def\date #1{\centerline {\large #1}}
\def\place #1{\begin{center}{\large #1}
\end{center}}
\let\Large=\large
\let\large=\normalsize

\newenvironment{proof}[1][Proof]{\noindent\textbf{#1.} }{\ \rule{0.5em}{0.5em}}

\begin{titlepage}
\def\thefootnote{\fnsymbol{footnote}}
\vspace*{1.1in}

\title{{\Large Information Percolation with Equilibrium Search Dynamics}\footnote{Duffie
is at the Graduate School of Business, Stanford University, a member of NBER, and acknowledges support while visiting
the Swiss Finance Institute at The University of Lausanne. Malamud is at the Department of
Mathematics, ETH Zurich and a member of the Swiss Finance
Institute. Manso is at the Sloan School of Business, MIT. We are
grateful for a conversation with Alain-Sol
Sznitman.}}

\vskip .75em

\author{Darrell Duffie, Semyon Malamud,  and Gustavo Manso}

\date{\today}
\vskip .75em
\vskip 1.90em
\baselineskip=.15in

\begin{abstract}
We solve for the equilibrium dynamics of information sharing in
a large population. Each agent is endowed
with signals regarding the likely outcome of a random variable
of common concern.  Individuals  choose the effort with which they search for others from whom
they can gather additional information.
When two agents meet, they share their information. The information gathered is
further shared at subsequent meetings, and so on.
Equilibria exist in which agents search maximally until they acquire
sufficient information precision, and then minimally.
A tax whose proceeds are used to subsidize the costs of search
improves information sharing and can in some cases increase
welfare. On the other hand, endowing agents with public signals
reduces information sharing and can in some cases decrease welfare.
\end{abstract}

\end{titlepage}

\setcounter{footnote}{0}

\section{Introduction}

We characterize the equilibrium dynamics of
information sharing in a large population.
An agent's optimal current
effort to search for information sharing opportunities depends on
that agent's current level of information and on the
cross-sectional distribution of information quality and search efforts
of other agents.  Under stated conditions, in equilibrium, agents search maximally
until their information quality reaches a trigger level, and then minimally.
In general, it is not the case that
 raising the search-effort policies of all agents causes an improvement
 in information sharing. This monotonicity property does, however, apply
 to trigger strategies, and enables a fixed-point algorithm for equilibria.

In our model, each member
of the population is endowed with signals regarding the likely
outcome of a Gaussian random variable $Y$ of common concern.  The ultimate
utility of each agent is increasing in the agent's conditional precision of $Y$.
Individuals therefore seek out others from whom they can gather additional
information about $Y$. When agents meet, they share their information.
The information gathered is then further shared at subsequent
meetings, and so on.  Agents meet according to a technology
for search and random matching, versions of which are
common in the economics literatures covering labor markets, monetary
theory, and financial  asset markets. A distinction is that the
search intensities in our model vary cross-sectionally
in a manner that depends  on the endogenously chosen efforts
of agents.

Going beyond prior work in this setting, we capture
implications of the incentive to search more intensively
whenever there is greater expected utility to be gained from the
associated improvement in the information arrival process.
Of course, the amount of information that can be
gained from others depends on the efforts that the
others themselves have made to search in the past. Moreover, the
current expected rate at which a given agent meets others
depends not only on the search efforts of that agent,
but also on the current search efforts of the others.
We assume complementarity in search efforts.
Specifically, we suppose that the intensity of arrival of matches
by a given agent increases in proportion to the current
search effort of that agent, given the search efforts
of the other agents. Each agent is modeled as fully
rational, in a sub-game-perfect Bayes-Nash equilibrium.

The existence and characterization of an equilibrium
involves incentive consistency conditions on the
jointly determined search efforts of all members of the
population simultaneously.
We find conditions for a stationary equilibrium, in which each agent's search
effort at a given time depends only on that agent's current level of
precision regarding the random variable $Y$ of common concern.
Each agent's life-time search intensity process is
the solution of a stochastic control problem, whose rewards
depend on the optimal search intensity processes of all other agents.

We show that if the cost of search is increasing and convex
in effort, then, taking as given the cross-sectional
distribution of  other agents'  information quality and search
efforts, the optimal search effort of any given agent is declining in the
current information precision of that agent. This property holds, even out of
equilibrium, because the marginal valuation of additional information for each
agent declines as that agent gathers additional information.
With proportional search costs, this property leads to equilibria with trigger
policies that reduce search efforts to a minimal level once a sufficient
amount of information is obtained.  Our proof of existence relies
on a monotonicity result: Raising the assumed trigger level at which all agents
reduce their search efforts leads to a first-order dominant cross-sectional distribution
of information arrivals.

We show by counterexample, however, that for general forms
of search-effort policies it is not generally true that the adoption of
more intensive search policies leads to an improvement
in population-wide information sharing.
Just the opposite can occur.  More
intensive search at given levels of information can in some cases advance the onset
of a reduction of the search efforts of agents who may be a rich source of
information to others. This can lower access to richly informed agents in such a manner
that, in some cases, information sharing is actually poorer.

The paper ends with an analysis of the welfare effects of policy interventions.
 First, we analyze welfare gains that can be achieved with a lump-sum tax whose proceeds are used to subsidize the costs of search efforts. Under stated conditions, we show that this promotes positive search
 externalities that would not otherwise arise in equilibrium.  Finally, we
show that, with proportional search costs, additional public information leads in equilibrium to an
unambiguous reduction in the sharing of private information, to the extent that there is in some cases a net negative welfare effect.

\section{Related Literature}

Previous research in economics has investigated the issue of
information aggregation. A large literature has focused on the
aggregation of information through prices. For example,
Grossman (1981) proposed the concept of rational-expectations
equilibrium to capture the idea that prices aggregate information
that is initially dispersed across investors. Wilson (1977),
Milgrom (1981), Pesendorfer and Swinkels (1997), and Reny
and Perry (2006) provide strategic foundations for the
rational-expectations equilibrium concept in centralized
markets.

In many situations, however, information aggregation occurs
through local interactions rather than through common observation
of  market prices. For example, in decentralized markets,
such as those for real estate and over-the-counter securities, agents
learn from the bids of other agents in private auctions or bargaining
sessions. Wolinsky (1990) and Blouin and Serrano (2002) study information
percolation in decentralized markets. In the literature on
social learning, agents communicate with each other and
choose actions based on information received from others.
Banerjee and Fudenberg (2004), for example, study
information aggregation in a social-learning context.

Previous literature has shown that some forms of information
externalities may slow down or prevent information aggregation.
For example, Vives (1993) and Amador and Weill (2007) show that
information aggregation may be slowed when agents
base their actions on public signals (price) rather
than on private signals, making inference noisier.
Bikhchandani, Hirshleifer, and Welch (1992) and Banerjee (1992)
show that agents may rely on publicly observed actions, ignoring
their private signals, giving rise to informational
cascades that prevent social learning.

Our paper studies information aggregation in a social learning
context. In contrast to previous studies, we analyze the equilibria
of a game in which agents seek out other agents from whom they can
gather information. This introduces a new source of information
externality. If an agent chooses a high search intensity, he
produces an indirect benefit to other agents by increasing both the
mean arrival rate at which the other agents will be matched and
receive additional information, as well as the amount of information
that the given agent is able to share when matched. We show that
because agents do not take this externality into account when
choosing their search intensities, social learning may be relatively
inefficient or even collapse.

In addition to the information externality problem, our paper shows
that coordination problems may be important in information
aggregation problems. In our model, there are multiple equilibria
that are Pareto-ranked in terms of the search intensity employed by
the agents. If agents believe that other agents are searching with
lower intensity, agents will also search with lower intensity,
producing an equilibrium with slower social learning.
Pareto-dominant equilibria, in which all agents search with higher
intensity, may be possible, but it is not clear how agents will
coordinate to achieve such equilibria.

Our technology of search and matching is similar to that used in
search-theoretic models that have provided foundations for
competitive general equilibrium and in models of equilibrium in
markets for labor, money, and financial assets.\footnote{Examples of
theoretical work using random matching to provide foundations for
competitive equilibrium include that of Rubinstein and Wolinsky
(1985) and Gale (1987). Examples in labor economics include
Pissarides(1985) and Mortensen (1986); examples in monetary theory
include Kiyotaki and Wright (1993) and Trejos and Wright (1995);
examples in finance include Duffie, G\^arleanu, and Pedersen (2005)
and Weill (2004).} Unlike these prior studies, we allow for
information asymmetry about a common-value component, with learning
from matching and with endogenously chosen search efforts.

Our model is related to that of Duffie and Manso (2007)
and Duffie, Giroux, and Manso (2008), which provide
an explicit solution for the evolution of posterior
beliefs when agents are randomly matched in groups
over time, exchanging their information with each
other when matched. In contrast to these prior studies,
however,  we model the endogenous choice of search intensities.
Moreover, we deal with
Gaussian uncertainty, as opposed to the case of
binary uncertainty that is the focus of these prior two papers.
Further, we allow for the entry and exit of agents,
and analyze the resulting stationary equilibria.

\section{Model Primitives}

A probability space $(\Omega, {\cal F},P)$ and a non-atomic measure
space $(A, {\cal A},\alpha)$ of agents are fixed.
We  rely  throughout on applications of the
exact law of large numbers (LLN) for a continuum of
random variables. A suitably precise version can
be found in Sun (2006), based on technical conditions
on the measurable subsets of $\Omega\times A$.  As
in the related literature, we also rely
formally on a continuous-time LLN for random search
and matching that has only been rigorously justified
in discrete-time settings.\footnote{See Duffie and Sun (2007).}
An alternative, which we avoid for simplicity, would be to
describe limiting results for a sequence of models with discrete time
periods or finitely many agents, as the lengths of time periods shrink
or as the number of agents gets large.

All agents benefit, in a manner to be explained, from information
about a particular random variable $Y$.
Agents are endowed with signals from a space ${\cal S}$.
The signals are jointly Gaussian with $Y$. Conditional on $Y$, the signals are pairwise
independent. We assume that $Y$ and all of the signals in ${\cal S}$
have zero mean and unit variance, which is without loss of generality
because they play purely informational roles.

Agent $i$  enters the market with a random number $N_{i0}$ of signals
that is independent of $Y$ and ${\cal S}$. The probability
distribution $\pi$ of $N_{i0}$ does not depend on $i$.  For
almost every pair $(i,j)$ of agents, $N_{i0}$ and $N_{j0}$ are
independent, and their signal sets are disjoint.

When present in the market, agents meet other agents according to
endogenous search and random matching dynamics to be described. Under these dynamics,
for almost every pair $(i,j)$ of agents, conditional on meeting at a given time $t$,
there is zero probability that they meet at any other time, and zero
probability that the set of agents that $i$ has met before $t$
overlaps with the set of
agents that $j$ has met before $t$.

Whenever two agents meet, they
share with each other enough information to reveal their respective
current conditional distributions of $Y$. Although we do not model
any strict incentive for matched agents to share their information,
they have no reason not to do so.  We could add to the model
a joint production decision that would provide a
strict incentive for agents to reveal their information when matched,
and have avoided this for simplicity.

By the joint Gaussian assumption, and by induction in the number of prior meetings that each has had,
it is enough that each of two matched agents tells the other his or her immediately prior
conditional mean and variance of $Y$.
The conditional variance
 of $Y$ given
any $n$ signals is
\[v(n)=
\frac{1-\rho^2}{1+\rho^2(n-1)},\]
where $\rho$ is the correlation between $Y$ and any signal. Thus,
 it is equivalent for the purpose of updating the agents'
conditional distributions of $Y$ that
agent $i$ tells his counterparty at any meeting at time $t$
his or her current conditional mean $X_{it}$ of $Y$
and the total number $N_{it}$ of signals that played a role
in calculating the agent's current conditional
distribution of $Y$. This number of signals is initially
the endowed number $N_{i0}$, and is then incremented at each
meeting by the number $N_{jt}$ of signals that similarly influenced
the information about $Y$ that had been gathered by his counterparty
$j$ at time $t$

Because the precision $1/v(N_{it})$ of the conditional
distribution of $Y$ given the information set ${\cal F}_{it}$ of
agent $i$ at time $t$ is strictly monotone in $N_{it}$,
we speak of  ``precision''  and $N_{it}$ interchangeably.

Agents remain in the market for exponentially distributed
times that are independent (pairwise) across agents, with parameter $\eta'$.
If exiting at time $t$,  agent $i$ chooses an action $A$, measurable
with respect to his current information ${\cal F}_{it}$, with cost $(Y-A)^2$.
Thus, in order to minimize the expectation of this cost, agent $i$ optimally chooses the action $A=E(Y\,\vert\, {\cal F}_{it})$, and incurs an optimal expected exit cost equal
to the ${\cal F}_{it}$-conditional variance $\sigma^2_{it}$ of $Y$.
Thus, while in the market, the agent has an incentive to gather information about $Y$
in order to reduce the expected exit cost. We will shortly explain how
search for other agents according to a costly effort process $\phi$
influences the current mean rate of arrival of matches, and thus the information filtration $\{{\cal F}_{it}: t\geq 0\}$.
Given a discount rate $r$, the agent's lifetime utility (measuring
time from the point of that agent's market entrance) is
\[U(\phi)=E\left(-e^{-r\tau} \sigma^2_{i\tau}-\int_0^\tau e^{-rt}K(\phi_t)\, dt\right),\]
where $\tau$ is the exit time and $K(c)$ is the cost rate for search effort level $c$, which
is chosen at each time from some interval $[c_L,c_H]\subset \R_+$.
We take the cost function $K$ to be bounded and measurable, so $U(\phi)$ is bounded
above and finite.

As we will show, essentially any exit utility formulation that is
concave and decreasing in $\sigma^2_{i\tau}$ would result
in precisely the same characterization of equilibrium that we shall provide.

The agent is randomly matched at a stochastic intensity that is proportional to
the current effort of the agent, given the efforts of other agents.
This proportionality assumption means that an agent who exerts search
effort $c$ at time $t$ has a current intensity (conditional mean
arrival rate) of $cbq_b$ of being matched to some agent from the set of
agents currently using effort level $b$ at time $t$, where $q_b$ is the current fraction of
the population using effort $b$. More generally, if the current cross-sectional
distribution of effort by other agents is given by a measure
$\epsilon$, then the intensity of a match with agents whose
current effort levels are in a set $B$ is $c\int_B b\, d\epsilon(b)$.
The particular pairings of counterparties are randomly chosen,
in the sense of the law of large numbers for pairwise random matching
of Duffie and Sun (2007).

Agents enter the market at a rate proportional to the current
mass $q_t$ of agents in the market, for some proportional ``birth rate'' $\eta>0$.
Because agents exit the market pairwise independently at intensity
$\eta'$, the law of large numbers implies that
the total quantity $q_t$ of agents in the market at time $t$ is $q_t=q_0e^{(\eta-\eta')t}$ almost surely.

The cross-sectional distribution $\mu_t$ of information precision at time $t$
is defined, at any set $B$ of positive integers, as the
fraction $\mu_t(B)=\alpha(\{i:N_{it}\in B\})/q_t$ of agents
whose precisions are currently in the set $B$.
We sometimes abuse notation by writing
$\mu_t(n)$ for the fraction of agents with precision $n$.

In the equilibria that we shall demonstrate, each agent
chooses an effort level at time $t$ that depends only on
that agent's current precision, according to a policy
$C:{\mathbb N}\rightarrow [c_L,c_H]$ used by all agents.
Assuming that such a search effort policy $C$ is used
by all agents, the cross-sectional precision distribution
satisfies (almost surely) the differential equation
\begin{equation}
\frac{d}{dt}\mu_t=\eta(\pi-\mu_t) + \mu^C_t*\mu^C_t-\mu^C_t\,\mu^C_t(\N),
\label{dynamic}
\end{equation}
where $\mu^C_t(n)=C_n\mu_t(n)$ is the effort-weighted
measure, $\mu*\nu$ denotes the convolution
of two measures $\mu$ and $\nu$, and
\[
\mu^C_t(\N)\ =\ \sum_{n=1}^\infty\,C_n\,\mu_t(n)
\]
is the cross-sectional average search effort.
The mean exit rate $\eta'$ plays no role in (\ref{dynamic}) because exit removes agents with a cross-sectional
distribution that is the same as the current population
cross-sectional distribution.
The first term on the right-hand side of (\ref{dynamic})
represents the replacement of agents with newly entering
agents. The convolution term $\mu^C_t*\mu^C_t$ represents
the gross rate at which new agents of a given precision
are created through matching and information sharing.
For example, agents of a given posterior precision $n$
can be created by pairing agents of prior respective
precisions $k$ and $n-k$, for any $k<n$, so the total gross rate of increase
of agents with precision $n$ from this source is
\[(\mu^C_t*\mu^C_t)(n)=\sum_{k=1}^{n-1}\mu_t(k)C(k)C(n-k)\mu_t(n-k).\]
The final term of (\ref{dynamic})
captures the rate $\mu^C_t(n)\,\mu_t^C(\N)\,$ of replacement of  agents with prior precision $n$
with those of some new posterior precision that  is obtained through matching and information sharing.

We  anticipate that, in each state of the world $\omega$ and at each time
$t$, the joint cross-sectional population distribution of precisions
and posterior means of $Y$ has a density $f_t$ on
${\mathbb N}\times {\mathbb R}$, with evaluation $f_t(n,x)$
at precision $n$ and posterior mean $x$ of $Y$. This means that
the fraction of agents whose conditional precision-mean pair $(n,x)$ is in
a given measurable set $B\subset  {\mathbb N}\times {\mathbb R}$
is $\sum_n\int_{-\infty}^{+\infty} f_t(n,x)1_{\{(n,x)\in B\}} \, dx.$
When it is important to clarify the dependence of this density on  the state
of world $\omega\in\Omega$, we write $f_t(n,x,\omega)$.

\begin{prop} For any search-effort policy function $C$, the
cross-sectional distribution $f_t$ of precisions and posterior means of the agents is
almost surely given by
\begin{equation}
\label{solution}
f_t(n,x,\omega)=\mu_t(n)\, p_n(x\, \vert\, Y(\omega)),
\end{equation}
where $\mu_t$ is the unique solution of the differential
equation (\ref{dynamic}) and $p_n(\,\cdot\, \vert\, Y)$ is the
$Y$-conditional Gaussian density of $E(Y\,\vert\, X_1,\ldots,X_n)$, for any $n$ signals
$X_1,\ldots,X_n$. This density has conditional mean
\[\frac{n\rho^2Y}{1+\rho^2(n-1)}\] and conditional variance
\begin{equation}
\label{sigman}
\sigma_n^2=
\frac{n\rho^2(1-\rho^2)}{(1+\rho^2(n-1))^2}.
\end{equation}
\label{ft}
\end{prop}

The appendix provides a proof based on a formal application of the law of large numbers,
and an independent proof by direct solution of the differential
equation for $f_t$ that arises from matching and information sharing.
As $n$ goes to infinity, the measure with density $p_n(\,\cdot\,
\vert\, Y)$ converges, $\omega$ by $\omega$ (almost surely), to a Dirac measure at $Y(\omega)$. In other
words, those agents that have collected a large number of signals have
posterior means that cluster (cross-sectionally) close to $Y$.

\section{Stationary Measure}
\label{measure}

In our eventual equilibrium, all agents adopt an optimal search effort policy function $C$, taking as given the presumption that
all other agents adopt the same policy $C$, and taking as given a stationary cross-sectional distribution
of posterior conditional distributions of $Y$.
Proposition \ref{ft} implies that this  cross-sectional distribution is determined
by the cross-sectional precision distribution $\mu_t$.
In a stationary setting, from (\ref{dynamic}), this precision distribution $\mu$
 solves
\begin{equation}
0=\eta(\pi-\mu)+\mu^C*\mu^C-\mu^C\,\mu^C(\N),
\label{stationary}
\end{equation}
which can be viewed as a form of algebraic Riccati equation.
We consider only solutions that have the correct total mass $\mu(\N)$ of 1.
For brevity, we use the notation $\mu_i\ =\ \mu(i)$ and $C_i\ =\ C(i)\,.$

\begin{lemma}\label{stat_mes} Given a policy $C$, there is a unique measure $\mu$
satisfying the stationary-measure equation (\ref{stationary}). This measure $\mu$ is characterized as follows. For any
$\bar C\in [c_L,c_H]$, construct a measure $\bar\mu(\bar C)$  by the algorithm:
\[
\bar \mu_1(\bar C)\ =\ \frac{\eta\,\pi_1}{\eta\ +\ C_1\,\bar C}
\]
and then, inductively,
\[
\bar \mu_k(\bar C)\ =\ \frac{\eta\,\pi_k\ +\ \sum_{l=1}^{k-1}\,C_l\,C_{k-l}\,\bar\mu_l(\bar C)\,\bar \mu_{k-l}(\bar C)}{\eta\ +\ C_k\,\bar C}.
\]
There is a unique solution $\bar C$ to the equation $\bar C=\sum_{n=1}^\infty \bar\mu_n(\bar C)\bar C$.
Given such a $\bar C$, we have $\mu=\bar \mu(\bar C)$.
\end{lemma}

  An important
question is stability. That is, if the initial condition $\mu_0$ is not
sufficiently near the stationary measure, will the solution path $\{\mu_t: t\geq 0\}$ converge to the
stationary measure?
The dynamic equation (\ref{dynamic}) is an
 infinite-dimensional non-linear dynamical system that could in principle have
potentially complicated oscillatory behavior.
In fact, a technical condition on the tail
behavior of the effort policy function $C(\,\cdot\,)$ implies that the stationary distribution is globally
attractive: From any initial condition, $\mu_t$ converges to the unique stationary
distribution.

\begin{prop}
Suppose that there is some integer $N$ such that  $C_n=C_N$
for $n\geq N$ and that $\eta\geq c_HC_N$.
Then the unique solution
$\mu_t$ of (\ref{dynamic}) converges pointwise  to the unique stationary measure $\mu$.
\label{combined}
\end{prop}

\noindent The proof, given in the appendix, is complicated by
the factor $\mu^C_t(\N)$, which is non-local and involves
$\mu_t(n)$ for each $n.$ The proof takes the approach
of representing the solution as a series $\{\mu_t(1),\mu_t(2),\ldots\}$, each term
of which solves an equation similar to (\ref{dynamic}), but
without the factor $\mu^C_t(\N).$ Convergence is proved for
each term of the series. A tail estimate completes
the proof.  The convergence of $\mu_t$
does not guarantee that the limit measure is in fact the unique stationary measure $\mu.$
The appendix includes a demonstration of this, based on Proposition \ref{nolostmass}.
As we later show in Proposition \ref{flattail},
the assumption that $C_n=C_N$ for all $n$ larger than some integer $N$
is implied merely by individual agent optimality, under a mild condition on
search costs.

Our eventual equilibrium will in fact be in the form of a trigger policy $C^N$, which for some integer
$N\geq 1$ is defined by
\begin{eqnarray*}
C^N_n&=& c_H, \quad n<N,\\
&=& c_L, \quad n\geq N.
\end{eqnarray*}
In other words, a trigger policy exerts maximal search effort until sufficient information precision is reached,
and then minimal search effort thereafter. A trigger policy automatically satisfies the ``flat tail''
condition of Proposition \ref{combined}.

A key issue is whether search policies that exert more effort at each precision
level actually generate more information sharing. This is an interesting question in its
own right, and also plays a role in obtaining a fixed-point proof of
existence of equilibria.  For a given
agent, access to information from others is entirely determined by
the weighted measure $\mu^C$, because if the given agent searches at some
rate $c$, then the arrival rate of agents that offer $n$ units of
additional precision is $cC_n\mu_n=c\mu^C_n.$ Thus, a first-order stochastically dominant (FOSD)
shift in the measure $\mu^C$ is an unambiguous improvement in the opportunity
of any agent to gather information. (A measure $\nu$ has the FOSD
dominant property relative to a measure $\theta$ if,  for any nonnegative
bounded increasing sequence $f$, we have $\sum_n f_n\nu_n\ge \sum_n f_n\theta_n.$)

 The next result states that, at least when comparing trigger policies, a more intensive search policy
 results in an improvement in information sharing opportunities.

\begin{prop} \label{M>N}
Let $\mu^M$ and $\nu^N$ be the unique stationary measures corresponding to trigger
policies $C^M$ and $C^N$ respectively. Let $\mu^{C,N}_n=\mu^N_nC^N_n$ denote the associated
search-effort-weighted measure. If $N>M$,
then $\mu^{C,N}$ has the first-order dominance property
over $\mu^{C,M}.$
\label{fosd}
\end{prop}

Superficially, this result may seem obvious. It says merely
that if all agents extend their high-intensity search to a higher level of precision, then there will be an unambiguous upward shift in the
cross-sectional distribution of information transmission rates.  Our proof, shown in the appendix,
is not simple. Indeed, we provide a counterexample below
to the similarly ``obvious'' conjecture that any increase in the common search-effort policy function leads to a first-order-dominant improvement in information sharing.

The issue of whether higher search efforts at each given
level of precision improves information sharing
involves two competing forces.
The direct effect is that higher search efforts at a given precision increases
the speed of information sharing, holding constant the precision distribution $\mu$.
The opposing effect is that if agents search harder, then they may earlier reach a
level of precision at which they reduce their search efforts,
which could in principle cause a downward shift in the cross-sectional average rate
of information arrival.
 In order to make precise these competing effects, we return
to the construction of the measure $\bar \mu(\bar C)$ in Lemma \ref{stat_mes},
and write $\bar \mu(C,\bar C)$ to show the dependence of this
candidate measure on the conjectured average search effort $\bar C$ as well as the given policy $C$.
We emphasize that $\mu$ is the stationary measure for $C$ provided $\mu=\bar\mu(C,\bar C)$ {\it and}
$\bar C=\sum_nC_n\mu_n$. From the algorithm stated by Lemma \ref{stat_mes},
 $\bar\mu_k(C,\bar C)$ is increasing in $C$ and
{\it decreasing} in $\bar C$, for all $k$. (A proof, by induction in $k$, is given in the appendix.)
Now the relevant question is: What effect does increasing $C$ have
on the stationary average search effort, $\bar C$, solving the equation $\bar C=
\sum_n\bar\mu_n(C,\bar C)C_n$? The following proposition shows that increasing
$C$ has a positive effect on $\bar C$, and thus through this channel, a negative effect on $\bar\mu_k(C,\bar C)$.

\begin{prop}\label{barc} Let $\mu$ and $\nu$ be the stationary measures associated with
policies $C$ and $D$. If $D\geq C$, then $\sum_nD_n\nu_n\geq \sum_nC_n\mu_n$.
That is, any increase in search policy increases the equilibrium average search effort.
\end{prop}

For trigger policies, the direct effect of increasing the search-effort policy $C$ dominates
the ``feedback'' effect on the cross sectional average rate of effort.
For other types of policies, this need not be the case, as shown by the following
counterexample, whose proof is given in the appendix.

\begin{example}
\label{counterexample}
Suppose that  $\,\pi_2>2\pi_1\,.$
Consider a policy $C$ with $C_n=0$ for $n\ge 3$. Fix $C_2>0$, and consider variation of $C_1$.
For $C_1$ sufficiently close to $C_2$, we show in the appendix that
\[
\sum_{k=2}^\infty\,C_k\mu_k\ =\ C_2\mu_2
\]
is monotone decreasing in $C_1\,.$
Thus, if we consider the increasing sequence
\begin{eqnarray*}
f_1\ &=& 0,\\
f_n\ &=& 1, \quad n\ge 2,
\end{eqnarray*}
we have $f\cdot \mu^C\ =\ C_2\,\mu_2\,$ strictly decreasing in $C_1$, for $C_1$ in a neighborhood of $C_2$, so we
do not have FOSD of $\mu^C$ with increasing $C$. In fact,
more search effort by those agents with precision $1$ can actually lead
to poorer information sharing. To see this, consider the policies
$D=(1,1,0,0,\ldots)$ and $C=(1-\epsilon,1,0,0,\ldots)$.
The measure $\mu^C$ has FOSD over the measure $\mu^D$ for any\footnote{For this,
we can without loss of generality take $f_1=1$ and calculate that
$h(\epsilon)=f\cdot \mu^C$ is decreasing in $\epsilon$ for sufficiently small $\epsilon>0$.}
 sufficiently small $\epsilon>0$.
 \end{example}

\section{Optimality}
\label{optimalitysection}

In this section, we study the optimal policy of a given agent
who presumes that precision is distributed in the population
according to some fixed measure $\mu$, and further presumes that
other agents search according to a
conjectured policy function $C$. We let $\overline C=\sum_nC_n\mu_n$ denote
the average search effort.

Given the conjectured market properties $(\mu,C)$, each agent $i$ chooses some
search-effort process $\phi:\Omega\times [0,\infty)\rightarrow
[c_L,c_H]$ that is progressively measurable with respect to that
agent's information filtration $\{{\cal F}_{it}: t\geq 0\}$, meaning that $\phi_t$ is based only on current information.
The posterior distribution of $Y$ given ${\cal F}_{it}$ has
conditional variance $v(N(t))$, where $N$ is the agent's
precision process and $v(n)$ is the variance of $Y$ given
any $n$ signals.

Assuming a discount rate $r$ on future expected benefits,
and given the conjectured market properties $(\mu,C)$,
an agent solves the problem
\begin{equation}
U(\phi)= \sup_{\phi}\, E\left(-e^{-r\tau}v(N^\phi_\tau)-\int_0^\tau e^{-st}K(\phi_t)\, dt \right),
\label{optimal}
\end{equation}
where $\tau$ is the time of exit, exponentially distributed
with parameter $\eta'$, and where the agent's precision
process $N^\phi$ is the pure-jump process with a given
initial condition $N_{0}$, with jump-arrival intensity
$\phi_t \overline C$,
and with jump-size probability distribution $\mu^C/\overline C$,
that is, with probability $C(j)\mu(j)/\overline C$ of jump size $j$.
We have abused notation by measuring calendar time for the agent
from the time of that agent's market entry.

For generality, we relax from this point the assumption that the
exit disutility is the conditional variance $v(N^\phi_\tau)$,
and allow the exit utility to be of the more general form
$u(N^\phi_\tau)$, for any bounded increasing
concave\footnote{We say that a real-valued
function $F$ on the integers is concave if $F(j+2)+F(j)\leq 2
F(j+1)$.}  function $u(\,\cdot\,)$ on the positive integers.
It can be checked that $u(n)=-v(n)$ is indeed a special case.

We say that $\phi^*$ is an optimal search effort process given
$(\mu,C)$ if $\phi^*$ attains the supremum (\ref{optimal}).
We further say that a policy function $\Gamma: {\mathbb N}\rightarrow
[c_L,c_H]$ is optimal given $(\mu,C)$ if the search effort
process $\{\Gamma(N_t):\, t\geq 0\}$ is optimal, where
the precision process $N$ uniquely satisfies the stochastic
differential equation with jump arrival intensity
$\Gamma(N_t)\overline C$ and with jump-size distribution $\mu^C/\overline C$.
(Because $\Gamma(n)$ is bounded by $c_H$, there is a unique solution $N$ to this stochastic
differential equation. See Protter (2005).)

We characterize agent optimality given $(\mu,C)$ using the principle
of dynamic programming, showing that the indirect utility, or ``value,'' $V_n$ for precision $n$ satisfies
the Hamilton-Jacobi-Bellman equation for optimal search effort
given by
\begin{equation}
\label{Bellman}
0\, = \ -\,(r+\eta')\,V_n\ +\ \eta' u_n +\,  \sup_{c\in [c_L,c_H]} \,
 \{-K(c)+ c \,\sum_{m=1}^\infty\,  (V_{n+m}-V_n) \mu^C_m\}.
 \end{equation}

A standard martingale-based verification argument for the following result
is found in the appendix.

\begin{lemma}
If $V$ is a bounded solution of the Hamilton-Jacobi-Bellman
equation (\ref{Bellman}) and the policy function $\Gamma$
satisfies the optimality condition that, for each $n$ and
for all $c\in [c_L,c_H]$,
\[
-K(\Gamma_n)+ \Gamma_n \,\sum_{m=1}^\infty\,  (V_{n+m}-V_n)
\mu^C_m \geq - K(c)+ c \,\sum_{m=1}^\infty\,  (V_{n+m}-V_n) \mu^C_m,\]
then $\Gamma$ is an optimal policy function given $(\mu,C)$, and $V_{N_0}$
is the value of this policy.
\label{verification}
\end{lemma}

We begin to lay out some of the properties of optimal policies, based
on conditions on the search-cost function $K(\,\cdot\,).$

\begin{prop}\label{decreasing_policy} Suppose that $K$ is increasing, convex,
and differentiable.
Then,  given $(\mu,C)$,
there is a policy $\Gamma$ that is optimal for all agents,
and the optimal search effort $\Gamma_n$ is monotone decreasing in the
current precision $n$.
\label{optimality}
\end{prop}

In order to calculate a precision threshold $\overline N$, independent of the measure $\mu,$
above which it is optimal to search minimally, we let $\,\overline u\,=\lim_nu(n)\,$,
which exists because $u_n$ is increasing in $n$ and bounded, and we let
\[
\overline N\ =\ \max\{n\, : \,c_H\,\eta'\,(r+\eta')\,(\overline u -\ u(n))\ \ge\ K'(c_L)\},
\]
which is finite if $K'(c_L)>0$. A proof of the following result is found in the appendix.

 \begin{prop} Suppose that $K(\,\cdot\,)$ is increasing, differentiable, and convex,
 with $K'(c_L)>0$.
 Then, for any optimal search-effort policy $\Gamma$,
 \[
 \Gamma_n\ =\ c_L, \quad n\ge \overline N.
 \]
 \label{flattail}
\end{prop}

In the special case of proportional and non-trivial search costs, it is in fact optimal
for all agents to adopt  a trigger policy, one that searches at maximal effort
until a trigger level of precision is reached, and at minimal effort thereafter.
This result, stated next, is a consequence of our prior results
that an optimal policy is decreasing and eventually reaches $c_L$,
and of the fact that with linear search costs, an optimal policy is ``bang-bang,''
therefore taking the maximal effort level $c_H$ at first, then eventually
switching to the minimal effort $c_L$ at a sufficiently high precision.

\begin{prop}\label{trigger_policy}
Suppose that $K(c)=\kappa c$ for some scalar $\kappa>0$.
 Then,  given $(\mu,C)$,
there is a trigger policy that is optimal for all agents.
\end{prop}

\section{Equilibrium}\label{sec_equil}

An equilibrium is a search-effort policy function $C$ satisfying:
(i) there is a unique stationary cross-sectional precision measure
$\mu$ satisfying the associated equation (\ref{stationary}), and (ii)
taking as given the market properties $(\mu, C)$, the policy function
$C$ is indeed optimal for each agent. Our main result is that, with
proportional search costs, there exists an equilibrium in the form
of a trigger policy.

\begin{theorem}
Suppose that $K(c)=\kappa c$ for some scalar $\kappa>0$. Then there exists
a trigger policy that is an equilibrium.
\label{main}
\end{theorem}

The theorem is proved using the following proposition and corollary.
We let $C^N$ be the trigger policy with trigger at precision level $N$, and we
let $\mu^N$ denote the associated stationary measure.
We let ${\cal N}(N)\subset {\mathbb N}$ be the set of trigger levels
that are optimal given the conjectured market properties
$(\mu^N,C^N)$ associated with a trigger level $N$.
We can look for an equilibrium in the form of a fixed point of
the optimal trigger-level correspondence ${\cal N}(\,\cdot\,)$,
that is, some $N$ such that $N \in {\cal N}(N)$.
The Theorem does not rely on the stability result
that from any initial condition, $\mu_t$ converges to $\mu$.
This stability applies, by Proposition \ref{combined}, provided that  $\eta\geq c_Hc_L$.

\begin{prop}
Suppose that $K(c)=\kappa c$ for some scalar $\kappa>0$.
Then ${\cal N}(N)$ is increasing in $N$, in the sense that if $N' \geq
N$ and if $k \in {\cal N}(N)$, then there exists some $k' \geq k$
in ${\cal N}(N')$. Further, there exists a uniform
upper bound on ${\cal N}(N)$, independent of $N$, given by
\[\overline N=\max \{j: \, c_H\eta'(r+\eta')(\overline u-u(j))\geq \kappa\}.\]
\label{nashoperator}
\end{prop}

Theorem \ref{main} then follows from:
\begin{cor}
The correspondence ${\cal N}$ has a fixed point $N$. An equilibrium is
given by the associated trigger policy $C^N$.
\end{cor}
\bigskip

Our proof, found in the appendix, leads to the following
algorithm for computing symmetric pure strategy
equilibria of the game. The algorithm finds all such equilibria in trigger
strategies.
\medskip

{\it \noindent\textbf{Algorithm:}
Start with $N = \bar N$.
\begin{enumerate}
\item Compute ${\cal N}(N)$.
If $N \in {\cal N}(N)$, then
output $C^N$ (an equilibrium
of the game). Go to the next step.
\item If $N > 0$, go back to Step 1
with $N=N-1$. Otherwise, quit.
\end{enumerate}}

There may exist multiple equilibria of the game. The following
proposition shows that the equilibria are Pareto-ranked
according to their associated trigger levels, and that
there is never ``too much'' search in equilibrium.

\begin{prop}
Suppose that $K(c)=\kappa c$ for some scalar $\kappa>0$.
If $C^N$ is an equilibrium of the game then it Pareto
dominates any outcome in which all agents employ a
policy $C^{N'}$ for a trigger level $N' < N$. In
particular, the set of equilibria of the game
is Pareto-ranked with equilibria associated
with higher trigger levels dominating
equilibria associated with lower trigger
levels.
\label{pareto}
\end{prop}

\subsection{Equilibria with Minimal Search}

We now consider conditions under which there are equilibria with
minimal search, corresponding to the trigger policy with trigger at $\,N=0.$
The idea is that such equilibria can arise because a presumption that other
agents make minimal search efforts can lead to a conjecture of such poor information sharing opportunities that any given
agent may not find it worthwhile to exert more than minimal search effort. We give
an explicit sufficient condition for such
equilibria, a special case of which is $c_L=0$. Clearly, with $c_L=0$, it is pointless for any agent to expend
any search effort if he or she assumes that all other agents make no effort to be found.

Let $\mu^0$ denote the stationary precision distribution associated with minimal search,
so that $\,\bar C\ =\ c_L$ is the average search effort. The value function $V$ of any agent solves
 \begin{equation}\label{hom_bel}
  (r\,+\,\eta'\,+\,c_L^2)\,V_n\ =\ \eta'\,u_n\ -\ K(c_L)\ +\ c_L^2\,\sum_{m=1}^\infty\,V_{n+m}\,\mu^0_m.
  \end{equation}
Consider the bounded increasing sequence $f$ given by
\[
f_n=\ (r\,+\,\eta'\,+\,c_L^2)^{-1}\,\left(\eta'\,u_n\,-\,K(c_L)\right).
\]
 Define the operator
$\,A\,$ on the space of bounded sequences by
\[
(A(g))_n\ =\ \frac{c_L^2}{r\,+\,\eta'\,+\,c_L^2}\,\sum_{m=1}^\infty\,g_{n+m}\,\mu^0_m.
\]

\begin{lemma}\label{I-A} The unique, bounded solution $\,V\,$ to \eqref{hom_bel} is given by
\[
V\ =\ (I\,-\,A)^{-1}\,(f)\ =\ \sum_{j=0}^\infty\,A^j\,(f),
\]
 which is concave and monotone increasing.
\end{lemma}

In order to provide simple conditions for minimal-search equilibria, let
\begin{equation}
\label{foc}
B\ =\ c_L\,\sum_{m=1}^\infty\,(V_{1\,+\,m}\ -\ V_1)\,\mu^0_m\ \ge\ 0.
\end{equation}

\begin{theorem}\label{homog} Suppose that $K(\,\cdot\,)$ is convex, increasing, and differentiable.
Then the minimal-search policy $C$, that with $C(n)=c_L$ for all $n$, is an equilibrium if and only if
$K'(c_L)\ \ge\ B.$
In particular, if $\,c_L\,=\,0\,,\,$ then $\,B\,=\,0\,$ and minimal search is always an equilibrium.
\end{theorem}

Intuitively, when the cost of search is small, there should exist equilibria with active search.

\begin{prop}\label{more0} Suppose that $K(c)\ =\ \kappa\,c\,$ and $c_L\ =\ 0\,.$ If $\pi_1\,>\,0\,$ and
\begin{equation}
\label{conditionK}
\kappa\ -\  \frac{\eta'\,(u(2)\,-\,u(1))\,c_H\mu_1^1}{r+\eta'}<0,
\end{equation}
then there exists an equilibrium trigger policy $C^N$ with $N\ge 1\,.$ This equilibrium strictly Pareto dominates the no-search equilibrium.
\end{prop}

\section{Policy Interventions}
\label{interventions}

In this section, we discuss the potential welfare implications of policy interventions. First,
we analyze the potential to improve welfare by a tax whose proceeds are used to subsidize the costs
of search efforts. This has the potential benefit of
positive search externalities that may not otherwise arise in equilibrium because
each agent does not search unless others are searching,
even though there are feasible search efforts that would make all agents
better off.

Then, we study the potentially
adverse implications of providing all entrants with some additional
common information. Although there is some direct benefit
of the additional information, we show that with proportional
search costs, additional public information leads to an
unambiguous reduction in the sharing of private information,
to the extent that there is in some cases a net negative welfare effect.

In both cases, welfare implications are judged in terms of the
utilities of agents as they enter the market. In this sense, the
welfare effect of an intervention is said to be positive it
improves the utility of every agent at the point in time that
the agent enters, and negative if it causes a reduction in the
utilities of all entering agents.

\subsection{Subsidizing Search}

A policy that may help to attenuate the negative welfare effects of
low information sharing is to institute a tax whose proceeds are
used to subsidize search costs.

We suppose for this purpose that
each agent pays a lump-sum tax $\tau$ at entry.
Search costs are assumed to be proportional, at rate $\kappa c$ for some $\kappa>0$.
Each agent is also offered a proportional
reduction $\delta$ in search costs, so that
the after-subsidy search cost function of each agent is $K_\delta(c) = (\kappa - \delta) c$.
The lump-sum tax has no effect on equilibrium search
behavior, so we can solve for an equilibrium policy $C$, as before, based on an after-subsidy
proportional search cost of $\kappa-\delta$.
Because of the law of large numbers, the total per-capita rate $\tau\eta$
of tax proceeds can then be equated to the total per-capita rate of subsidy
by setting
\[\tau=\frac{1}{\eta}\, \delta\sum_n\mu_nC_n.\]

The search subsidy can potentially improve welfare by
addressing the failure, in a low-search
equilibrium,
to exploit positive
search externalities. As Proposition \ref{pareto} shows, there is
never too much search in equilibrium. The following lemma
and proposition show that, indeed,
equilibrium search effort is increasing in the search subsidy
rate $\delta$.

\begin{lemma}\label{mon_subsidy} Suppose that $K(c)=\kappa c$ for
some $\kappa >0$. For given market conditions $(\mu,C),$ the
trigger level $N$ in the precision of an optimal policy $C^N$ is
increasing in the search subsidy rate $\delta$. That is, if $N$ is an
optimal trigger level of precision given a subsidy $\delta$,
then for any higher search subsidy $\delta'\geq \delta$, there
exists a higher optimal trigger $N'\geq N$.
\end{lemma}

Coupled with Proposition \ref{M>N}, this lemma implies that an increase
in the subsidy allows an increase (in the sense of first order
dominance) in information sharing. A direct consequence of this lemma
is the following result.

\begin{prop}\label{equil_subsidy}
Suppose that $K(c)=\kappa c$ for some $\kappa >0$. If $C^N$ is an
equilibrium with proportional search subsidy $\delta$, then for any $\delta'\geq \delta$,
there exists some $N' \geq N$ such that $C^{N'}$ is an equilibrium
with proportional search subsidy $\delta'$.
\end{prop}

\paragraph{Example.} Suppose, for some integer $N>1$, that $\pi_0 = 1/2$, $\pi_N = 1/2$,
and $c_L = 0$. This setting
is equivalent to that of Proposition \ref{more0}, after noting
that every information transfer is in blocks of $N$ signals each,
resulting in a model isomorphic to one in which each
agent is endowed with one private signal of a particular
higher correlation. Recalling that inequality (\ref{conditionK})
determines whether zero search is optimal, we
 can exploit continuity of the lefthand side of this inequality
 to choose parameters so that, given market
conditions $(\mu^N,C^N)$, agents have a strictly positive but arbitrarily small increase in utility
when choosing search
policy $C^{0}$ over policy $C^N$. With this, $C^0$ is the
unique equilibrium. This is before considering a search subsidy.
We now consider a model
that is identical with the exception that each agent
is taxed at entry and given search subsidies
at the proportional rate $\delta$.
We can choose $\delta$ so that all agents
strictly prefer $C^N$ to $C^0$ (the strict no-search
condition (\ref{conditionK}) is satisfied),
and $C^N$ is an equilibrium. For sufficiently
large $N$ all agents have strictly higher indirect
utility in the equilibrium with the search subsidy
than they do in the equilibrium with the
same private-signal endowments and no
subsidy.

\subsection{Educating Agents at Birth}

A policy that might in principle attenuate the negative welfare effects of
low information sharing is to
``educate'' all agents, by giving all agents additional
public signals at entry.
We assume for simplicity that the $M\geq 1$ additional public signals are drawn from the same signal set ${\cal S}$.
When two agents meet and share information, they take into
account that the information reported by the other agent
contains the effect of the additional public signals. (The implications of the reported conditional mean and variance for the conditional mean and variance
associated with a counterparty's  non-public information can be inferred from the public
signals, using Lemma \ref{first}.)
Because of this, our prior analysis of information sharing dynamics
can be applied without alteration, merely by treating the precision
level of a given agent as the total precision less the public precision,
and by treating the exit utility of each agent for $n$ non-public signals as $\hat u(n)=u(n+M)$.

The public signals influence optimal search efforts.
Given the market conditions $(\mu,C)$, the indirect utility $V_n$ for non-public precision
$n$ satisfies the Hamilton-Jacobi-Bellman equation for optimal
search effort given by
\begin{equation}
\label{Bellmanpublic}
0\, = \ -\,(r+\eta')\,V_n\ +\ \eta' u_{M+n} +\,  \sup_{c\,\in\, [c_L,c_H]} \,
\big \{-K(c)+ c \,\sum_{m=1}^\infty\,  (V_{n+m}-V_n) \mu^C_m\big\}.
 \end{equation}

Educating agents at entry with public signals has two effects. On one hand,
when agents enter the market they are better informed
than if they had not received the extra signals. On the
other hand, this extra information may reduce agents' incentives
to search for more information, slowing down information
percolation. Below, we show an example in which the net effect is a strict welfare loss.
First, however, we establish that adding public information causes an
unambiguous reduction in the sharing of private information.

\begin{lemma}\label{mon_educ} Suppose that $K(c)=\kappa c$ for some $\kappa >0$. For given market
conditions $(\mu,C),$ the trigger level $N$ in non-public precision of an optimal policy $C^N$
 is decreasing in the precision $M$ of the public signals. (That is, if $N$ is an optimal trigger level
 of precision given public-signal precision $M$,
 then for any higher public precision $M'\geq M$, there exists a lower optimal trigger $N'\leq N$.)
\end{lemma}

Coupled with Proposition \ref{M>N}, this lemma implies that adding public information leads to a reduction
(in the sense of first order dominance) in information sharing.
A direct consequence of this lemma is the following result.

\begin{prop}\label{equil_educ}
Suppose that $K(c)=\kappa c$ for some $\kappa >0$. If $C^N$ is an equilibrium with $M$ public signals,
then for any $M'\leq M$, there exists some $N' \geq N$ such that $C^{N'}$ is
an equilibrium with $M'$ public signals.
\end{prop}

In particular, by removing all public signals, as in the following example, we can
get strictly superior information sharing, and in some cases a strict welfare improvement.

\paragraph{Example.} As in the previous example, suppose, for some integer $N>1$, that $\pi_0 = 1/2$, $\pi_N = 1/2$,
and $c_L = 0$. This setting
is equivalent to that of Proposition \ref{more0}, after noting
that every information transfer is in blocks of $N$ signals each,
resulting in a model isomorphic to one in which each
agent is endowed with one private signal of a particular
higher correlation. Analogously with the previous example, we can
exploit continuity  in the model parameters of the lefthand side of inequality (\ref{conditionK}), determining whether
zero search is optimal, to
choose the parameters so that, given
market conditions $(\mu^N,C^N),$ agents have a strict but arbitrarily small preference
 of policy $C^N$ over $C^{0}$.
We now consider a model that is
identical with the exception that each agent
is given $M=1$ public signal at entry.
With this public signal, again using continuity we can choose parameters so that all agents strictly prefer
$C^0$ to $C^N$ (the strict no-search
condition (\ref{conditionK}) is satisfied), and $C^0$ is the only equilibrium.
For sufficiently large $N$, or equivalently for any $N\geq 2$ and sufficiently small
signal correlation $\rho$, all agents have strictly lower indirect utility
in the equilibrium with the public signal at entry than they do in the equilibrium
with the same private-signal endowments and no public signal.

\newpage

\bigskip
\appendix

\bigskip

\centerline{\LARGE \bf Appendices}

\bigskip

\section{Proofs for Section 3: Information Sharing Model}

\begin{lemma} Suppose that $Y,X_1,\ldots,X_n,Z_1,\ldots,Z_m$ are joint
Gaussian, and that $X_1,\ldots,X_n$ and $Z_1,\ldots,Z_m$ all have
correlation $\rho$ with $Y$ and are $Y$-conditionally $iid$.
Then
\[E(Y\,\vert\, X_1,\ldots,X_n,Z_1,\ldots,Z_m)= \frac{\gamma_n}{\gamma_{n+m}}E(Y\,\vert\, X_1,\ldots,X_n)+\frac{\gamma_m}{\gamma_{m+n}}
E(Y\,\vert\, Z_1,\ldots,Z_m),\]
where $\gamma_k=1+\rho^2(k-1)$.
\label{first}
\end{lemma}

\begin{proof}
The proof is by calculation. If $(Y,W)$ are joint mean-zero Gaussian
and $W$ has an invertible covariance matrix, then by a well known result,
\[E(Y\,\vert\,W)=W^\top{\rm cov}(W)^{-1}{\rm cov}(Y,W).\]
It follows by calculation that
\begin{equation}
\label{CM}
E(Y\,\vert\, X_1,\ldots,X_n)=\beta_n(X_1+\cdots+X_n),
\end{equation}
where
\[\beta_n=\frac{\rho}{1+\rho^2(n-1)}.
\]
Likewise,
\begin{eqnarray*}
E(Y\,\vert\, X_1,\ldots,X_n,Z_1,\ldots,Z_m)&=& \beta_{n+m}(X_1+\cdots+X_n+Z_1+\cdots + Z_m)\\
&=& \beta_{n+m}\left(\frac{E(Y\,\vert\, X_1,\ldots,X_n)}{\beta_n}+\frac{E(Y\,\vert\, Z_1,\ldots,Z_m)}{\beta_m}\right).
\end{eqnarray*}
The result follows from the fact that $\beta_{n+m}/\beta_n=\gamma_n/\gamma_{n+m}$.
\end{proof}
\bigskip

\begin{cor}
\label{pn}
The conditional probability density $p_n(\,\cdot\, \vert\, Y)$
of $E(Y\,\vert\, X_1,\ldots,X_n)$ given $Y$ is almost surely
Gaussian with conditional mean
\[\frac{n\rho^2Y}{1+\rho^2(n-1)}\] and with conditional variance
\begin{equation}
\sigma_n^2=
\frac{n\rho^2(1-\rho^2)}{(1+\rho^2(n-1))^2}.
\end{equation}
\end{cor}

\noindent {\bf Proof of Proposition \ref{ft}.}

We use the conditional law of large numbers (LLN) to calculate the
cross-sectional population density $f_t$.
Later, we independently calculate $f_t$, given the appropriate
boundary condition $f_0$,  by a direct solution of the
particular dynamic equation that arises from updating
beliefs at matching times.

Taking the first, more abstract, approach, we fix a
time $t$ and state of the world $\omega,$ and let $W_n(\omega)$ denote the
set of all agents whose current precision is $n$.
We note that $W_n(\omega)$ depends non-trivially on $\omega$. This set
$W_n(\omega)$ has an infinite number of agents whenever $\mu_t(n)$ is
non-zero, because the space of agents is non-atomic. In particular,
the restriction of the measure on agents to $W_n(\omega)$ is non-atomic.
Agent $i$ from this set $W_n(\omega)$ has a current conditional mean of $Y$ that is
denoted $U_i(\omega)$. Now consider the cross-sectional distribution,
$q_n(\omega)$, a measure on the real line, of $\{U_i(\omega) : i \in
W_n(\omega)\}$. Note that the random variables $U_i$ and
$U_j$ are $Y$-conditionally independent for almost every distinct pair
$(i,j)$, by the random matching model, which implies by
induction in the number of their finitely many prior
meetings that they have conditioned on distinct subsets of signals, and that
the only source of correlation in $U_i$ and $U_j$ is the fact that each of these
posteriors is a linear combination of $Y$ and of other pairwise-independent variables
that are also jointly independent of $Y.$

Conditional on the event $\{N_{it}=n\}$ that agent $i$ is in the set $W_n(\omega),$ and conditional on $Y$, $U_i$ has the
Gaussian conditional density $p_n(\,\cdot\, \vert\, Y)$ recorded in Corollary \ref{pn}.
This conditional density function does not depend on $i$.
Thus, by a formal application of the law of large numbers, in almost every state of the world $\omega$,
$q_n(\omega)$ has the same distribution as the
$(W_n,Y)$-conditional distribution of $U_i$, for any $i$.
Thus, for almost every $\omega,$ the cross-sectional
distribution $q_n(\omega)$ of posteriors over
the subset $W_n(\omega)$ of agents has the density
$p_n(\,\cdot\, \vert\, Y(\omega))$.
In summary, for almost every state of the world,
the fraction $\mu_t(n)$ of the population that has received $n$
signals has a cross-sectional density $p_n(\,\cdot\, \vert\,
Y(\omega))$ over their posteriors for $Y$.

We found it instructive to consider a more concrete proof based on a computation of the solution of the
appropriate differential equation for $f_t$, using the LLN to
set the initial condition $f_0$.  Lemma \ref{first} implies that when an agent with
joint type $(n,x)$ exchanges all information with an agent whose type
is $(m,y)$, both agents achieve posterior type
\[\left(m+n,\frac{\gamma_n}{\gamma_{m+n}}x+\frac{\gamma_m}{\gamma_{m+n}}y\right).\] We therefore have
the dynamic equation
\begin{equation}
\label{evolve}
\frac{d}{dt}f_t(n,x)\ =\ \eta(\Pi(n,x)\ \ -\ f_t(n,x))\  +\ (f_t\circ f_t)(n,x) - \, C_nf_t(n,x)\sum_{m=1}^\infty\,C_m\,\int_\R\,f_t(m,x)\, dx\, ,
\end{equation}
where $\Pi(n,x)= \pi(n)\,p_n(x\,|\,Y(\go))$ and
\[(f_t\circ f_t)(n,x)=\sum_{m=1}^{n-1}\,\frac{\gamma_n}{\gamma_{n-m}}\,C_{n-m}\,C_m\,\int_{-\infty}^{+\infty}f_t\left(n-m,\frac{\gamma_{n}x-\gamma_my}{\gamma_{n-m}}\right)f_t(m,y)\, dy.
\]

It remains to solve this ODE for $f_t$. We will use the following calculation.

\begin{lemma}\label{lll} Let $\,q_1(x)\,$ and  $\,q_2(x)\,$ be the Gaussian
densities with respective means $\,M_1\,,\,M_2\,$ and variances $\,\gs_1^2\,,\,\gs_2^2\,.$ Then,
\[
\frac{\gamma_n}{\gamma_{n-m}}\,\int_{-\infty}^{+\infty}\,q_1\left(\frac{\gamma_{n}x-\gamma_my}{\gamma_{n-m}}\right)q_2(y)\, dy\ =\ q(x),
\]
where $\,q(x)\,$ is the density of a Gaussian with mean
\[
 M \ =\ \frac{\gamma_{n-m}}{\gamma_n}\,\mu_1\ +\ \frac{\gamma_m}{\gamma_n}\,\mu_2\
\]
and variance
\[
\gs^2\ =\ \frac{\gamma_{n-m}^2}{\gamma_n^2}\,\gs_1^2\ +\ \frac{\gamma_m^2}{\gamma_n^2}\,\gs_2^2.\
\]
\end{lemma}

\begin{proof} Let $\,X\,$ be a random variable with density
$\,q_1(x)\,$ and $\,Y\,$ an independent variable with
density $\,q_2(x)\,.$ Then
\[
Z\ =\ \gamma_n^{-1}\,\left(\,\gamma_{n-m}\,X\ +\ \gamma_m\,Y\,\right)
\]
is also normal with mean $M$ and variance $\,\gs^2\,.$ On the other hand, $\gamma_n^{-1}\,\gamma_{n-m}\,X\,$ and $\,\gamma_n^{-1}\,\gamma_m\,Y\,$ are independent with densities
\[
\,\frac{\gamma_n}{\gamma_{n-m}}\,q_1\left(\frac{\gamma_n}{\gamma_{n-m}}\,x\right)\,
\]
and
\[
\,\frac{\gamma_n}{\gamma_{m}}\,q_2\left(\frac{\gamma_n}{\gamma_{m}}\,x\right),\,
\]
respectively.
Consequently, the density of $\,Z\,$ is the convolution
\begin{multline}
\frac{\gamma_n^2}{\gamma_{n-m}\,\gamma_{m}}\,\int_\R\,q_1\left(\frac{\gamma_n}{\gamma_{n-m}}\,(x\,-\,y)\,\right)\,q_2\left(\frac{\gamma_n}{\gamma_{m}}\,y\right)\,dy\ \\=\  \frac{\gamma_n}{\gamma_{n-m}}\,\int_{-\infty}^{+\infty}\,q_1\left(\frac{\gamma_{n}x-\gamma_my}{\gamma_{n-m}}\,,\,\gs_{n-m}\right)q_2(z)\, dz,
\end{multline}
where we have made the transformation $\,z\ =\ \gamma_n\,\gamma_m^{-1}\,y\,.$
\end{proof}

\begin{lemma} The density
\[
f_t(n,x,\go)\ =\ \mu_t(n)\,p_n(x|Y(\go))
\]
solves the evolution equation \eqref{evolve} if and only if the
 distribution $\mu_t$ of precisions solves the evolution equation \eqref{dynamic}.
\end{lemma}

\begin{proof} By Lemma \ref{lll} and Corollary \ref{pn},
\[
\frac{\gamma_n}{\gamma_{n-m}}\int_{-\infty}^{+\infty}p_{n-m}\left(\frac{\gamma_{n}x-\gamma_my}{\gamma_{n-m}}\,\big|\,Y(\go)\,\right)p_m(y\,|\,Y(\go))\, dy\
\]
is conditionally Gaussian with mean
\[
\frac{\gamma_{n-m}}{\gamma_n}\,\frac{(n-m)\rho^2Y}{1+\rho^2(n-m-1)}\,+\,\frac{\gamma_m}{\gamma_n}\,\frac{m\rho^2Y}{1+\rho^2(m-1)}\ =\ \frac{n\rho^2Y}{1+\rho^2(n-1)}
\]
and conditional variance
\[
\gs^2\ =\ \frac{\gamma_{n-m}^2}{\gamma_n^2}\,\frac{(n-m)\rho^2(1-\rho^2)}{(1+\rho^2(n-m-1))^2}
\ +\ \frac{\gamma_m^2}{\gamma_n^2}\,\gs_2^2\,\frac{m\rho^2(1-\rho^2)}{(1+\rho^2(m-1))^2}\ =\ \frac{n\rho^2(1-\rho^2)}{(1+\rho^2(n-1))^2}.
\]
Therefore,
\begin{eqnarray*}
\!\!\!\!\!\!\!\!\!\!\!\!\!\!\!\!\!\!\!\!\!\!\!\!\!\!\!\!\!\!\!\!\!\!\!\!\!
&&\!\!\!\!\!\!\!\!\!\!\!\!\!(f_t\circ f_t)(n,x)\\
&=&\sum_{m=1}^{n-1}\,\frac{\gamma_n}{\gamma_{n-m}}\,C_{n-m}\,C_m\,\int_{-\infty}^{+\infty}f_t\left(n-m,\frac{\gamma_{n}x-\gamma_my}{\gamma_{n-m}}\right)f_t(m,y)\, dy\ \\
&=&\ \sum_{m=1}^{n-1}\,\,C_{n-m}\,C_m\,\mu_t(n-m)\,\mu_t(m)\,\frac{\gamma_n}{\gamma_{n-m}}\int_{-\infty}^{+\infty}p_{n-m}\left(\frac{\gamma_{n}x-\gamma_my}{\gamma_{n-m}}\,\big|\,Y(\go)\,\right)p_m(y\,|\,Y(\go))\, dy\ \\
&=&\
\sum_{m=1}^{n-1}\,\,C_{n-m}\,C_m\,\mu_t(n-m)\,\mu_t(m)\,p_n(x\,|\,Y(\go)).
\end{eqnarray*}
Substituting the last identity into \eqref{evolve}, we get the required result.
\end{proof}

\bigskip

\section{Proofs for Section 4: Stationary Distributions}

This appendix provides proofs of the results on the existence, stability, and monotonicity
properties of the stationary
cross-sectional precision measure $\mu$.

\subsection{Existence of the stationary measure}

\begin{proof}[Proof of Lemma \ref{stat_mes}]  If a positive, summable sequence $\{\mu_n\}$ indeed solves
\eqref{stationary}, then, adding up the equations over $n$, we get that $\mu(\N)=1,$ that is, $\mu$ is indeed a  probability measure. Thus, it remains to show that the equation
\[
\bar C\ =\ \sum_{n=1}^\infty \bar\mu_n(\bar C)C_n
\]
has a unique solution. By construction, the function $\bar \mu_k(\bar C)$ is monotone decreasing in $\bar C\,,$
and
\[
\eta\,\bar\mu_k(\bar C)\ =\ \eta\pi_k\ +\ \sum_{l=1}^{k-1}C_lC_{k-l}\bar\mu_l(\bar C)\bar\mu_{k-l}(\bar C)\ -\ C_k\,\bar\mu_k(\bar C)\,\bar C.
\]
Clearly, $\bar\mu_1(\bar C)\ <\ \pi_1\le1.$
Suppose that $\bar C\ge\ c_H\,.$ Then, adding up the above identities, we get
\[
\eta\,\sum_{k=1}^n\,\mu_k(\bar C)\ \le\ \eta\ +\ c_H\,\sum_{l=1}^{k-1}C_l\,\bar \mu_l\ -\ \bar C\sum_{l=1}^{k-1}\,C_l\bar\mu_l\ \le\ \eta.
\]
Hence, for $\bar C\ge\ c_H$ we have that
\[
\sum_{k=1}^\infty\,\mu_k(\bar C)\ \le\ 1.
\]
Consequently, the function
\[
f(\bar C)\ =\ \sum_{k=1}^\infty\,C_k\,\bar\mu_k(\bar C)
\]
is strictly monotone decreasing in $\bar C$ and satisfies
\[
f(\bar C)\ \le \ \bar C, \quad \bar C\ge c_H\,.\]
It may happen that $\,f(x)\ =\ +\infty$ for some $C_{\min}\,\in\,(0,\bar C)\,.$ Otherwise, we set $C_{\min}=0\,.$ The function
\[
g(\bar C)\ =\ \bar C\ -\ f(\bar C)
\]
is continuous (by the monotone convergence theorem for infinite series (see, e.g., Yeh (2006), p.168)) and strictly monotone increasing and satisfies $g(C_{\min})\ \le\ 0$ and $g(c_H)\ \ge\ 0\,.$ Hence, it has a unique zero.
\end{proof}

\subsection{Stability of the stationary measure}

\bigskip\noindent
{\bf Proof of Proposition
\ref{combined}.} \quad  The ordinary differential equation for $\mu_k(t)$ can be written as
\begin{equation}\label{eqq}
\mu_k'\ =\ \eta\,\pi_k\ -\ \eta\,\mu_k\ -\ C_k\,\mu_k\,\sum_{i=1}^\infty\,C_i\,\mu_i\ +\ \sum_{l=1}^{k-1}\,C_l\mu_l\,C_{k-l}\mu_{k-l}.
\end{equation}
We will need a right to interchange infinite summation and differentiation. We will use the following known

\begin{lemma}\label{lemconv}  Let $\,g_k(t)\,$ be $\,C^1$ functions such that
\[
\sum_k\,g'_k(t)\qquad and \qquad \sum_k\,g_k(0)
\]
converge for all $\,t\,$ and
\[
\sum_k|g_k'(t)|\
\]
is locally bounded (in $t$). Then, $\sum_k\,g_k(t)$ is differentiable and
\[
\left(\,\sum_k\,g_k(t)\,\right)'\ =\ \sum_k\,g_k'(t).
\]
\end{lemma}

We will also need

\begin{lemma}\label{linODE} Suppose that $f$ solves
\[
f'\ =\ -\,a(t)\,f\ +\ b(t),
\]
where $\,a(t)\,\ge\,\eps\,>\,0$ and
\[
\lim_{t\to\infty}\,\frac{b(t)}{a(t)}\ =\ c.
\]
Then,
\[
f(t)\ =\ e^{-\int_0^t\,a(s)ds}\,\int_0^t\,e^{\int_0^s\,a(u)du}\,b(s)ds\ +\ f_0\,e^{-\int_0^t\,a(s)ds}
\]
and
$
\lim_{t\to\infty}\,f(t)\ =\ c.
$
\end{lemma}

\begin{proof} The formula for the solution is well known. By l'H\^opital's rule,
\[
\lim_{t\to\infty}\,\frac{\int_0^t\,e^{\int_0^s\,a(u)du}\,b(s)ds}{e^{\int_0^t\,a(s)ds}}\ =\
\lim_{t\to\infty}\,\frac{e^{\int_0^t\,a(s)ds}\,b(t)ds}{a(t)\,e^{\int_0^t\,a(s)ds}}\ =\ c.
\]
\end{proof}

The following proposition shows existence and uniqueness of the solution.

\begin{prop}\label{exun} There exists a unique solution $\{\mu_k(t)\}$ to \eqref{eqq} and this solution satisfies
\begin{equation}\label{mass}
\sum_{k=1}^\infty\,\mu_k(t)\ =\ 1
\end{equation}
for all $t\ge 0.$
\end{prop}

\begin{proof} Let $l_1(\N)$ be the space of absolutely summable sequences $\{\mu_k\}$ with
\[
\|\{\mu_k\}\|_{l_1(\N)}\ =\ \sum_{k=1}^\infty\,|\mu_k|.
\]
Consider the mapping $F:\ l_1(\N)\ \to\ l_1(\N)$ defined by
\[
(F(\{\mu_i\}))_k\ =\ \eta\,\pi_k\ -\ \eta\,\mu_k\ -\ C_k\,\mu_k\,\sum_{i=1}^\infty\,C_i\,\mu_i\ +\ \sum_{l=1}^{k-1}\,C_l\mu_l\,C_{k-l}\mu_{k-l}.
\]
Then, \eqref{eqq} takes the form $(\{\mu_k\})'\ =\ F(\{\mu_k\}).$
A direct calculation shows that
\begin{multline}
\sum_{k=1}^\infty\,\left|\sum_{l=1}^{k-1}\,(C_la_l\,C_{k-l}a_{k-l}\ -\ C_lb_l\,C_{k-l}b_{k-l})\right|\  \\ \le\ c_H^2\,\sum_{k=1}^\infty\,\sum_{l=1}^{k-1}(|a_l-b_l|\,|a_{k-l}|\ +\ |a_{k-l}-b_{k-l}|\,|b_l|\ \\
\!\!\!\!\!\!\!\!\!\!\!\!\!\!\!\!\!\!\!\!\!\!\!\!\!\!\!\!\!\!\!\!\!\!\!\!\!\!\!\!\!\!\!\!\!\!\!\!\!\!\!\!\!\!\!\!\!\!\!\!\!\!\!\!\!\!\!\!\!\!\!\!\!\!\!\!\!\!\!\!\!\!\!\!\!\!\!\!\!\!\!\!
=\ c_H^2(\|\{a_k\}\|_{l_1(\N)}+\|\{b_k\}\|_{l_1(\N)})\,\|\{a_k-b_k\}\|_{l_1(\N)}.
\end{multline}
Thus,
\[
\|F(\{a_k\})\ -\ F(\{b_k\})\|_{l_1(\N)}\ \le\ (\eta\ +\ 2\,c_H^2\,(\|\{a_k\}\|_{l_1(\N)}\,+\|\{b_k\}\|_{l_1(\N)})\,
\
)
\|\{a_k-b_k\}\|_{l_1(\N)},
\]
so $F$ is locally Lipschitz continuous. By a standard existence
result (Dieudonn\'e (1960), Theorem 10.4.5), there exists a unique
solution to \eqref{eqq} for $t\in[0,T_0)$ for some $T_0>0$ and this
solution is locally bounded. Furthermore, $[0,T_0)$ can be chosen to
be the maximal existence interval, such that the solution
$\{\mu_k\}$ cannot be continued further. It remains to show that
$T_0=+\infty.$ Because, for any $t\in[0,T_0),$
\[
\|(\{\mu_k\})'\|_{l_1(\N)}\ =\ \|F(\{\mu_k\})\|_{l_1(\N)}\ \le\ \eta+\eta\|\{\mu_k\}\|_{l_1(\N)}\ +\ 2c_H^2\|\{\mu_k\}\|_{l_1(\N)}^2
\]
is locally bounded, Lemma \ref{lemconv} implies that
\[
\left(\sum_{i=1}^\infty\,\mu_k\right)'\ =\ \sum_{k=1}^\infty\left(
\eta\,\pi_k\ -\ \eta\,\mu_k\ -\ C_k\,\mu_k\,\sum_{i=1}^\infty\,C_i\,\mu_i\ +\ \sum_{l=1}^{k-1}\,C_l\mu_l\,C_{k-l}\mu_{k-l}
\right)\ =\ 0,
\]
and hence \eqref{mass} holds. We will now show that $\mu_k(t)\ge 0$ for all $t\in[0,T_0].$ For $k=1,$ we have
\[
\mu_1'\ =\ \eta\pi_1\ -\ \eta\mu_1\ -\ C_1\mu_1\sum_{i=1}^\infty C_i\mu_i.
\]
Denote $a_1(t)\ =\ -\eta-C_1\sum_{i=1}^\infty C_i\mu_i.$ Then, we have
\[
\mu_1'\ =\ \eta\pi_1\ +\ a_1(t)\mu_1.
\]
Lemma \ref{linODE} implies that $\mu_1\ge 0$ for all $t\in[0,T_0).$ Suppose we know that $\mu_l\ge 0$ for $l\le k-1.$ Then,
\[
\mu_k'\ =\ z_k(t)\ +\ a_k(t)\,\mu_k(t)\ ,\ a_k(t)\ =\ -\eta-C_k\sum_{i=1}^\infty C_i\mu_i\ ,\ z_k(t)\ =\ \eta\pi_k+\sum_{l=1}^{k-1}\,C_l\mu_l\,C_{k-l}\mu_{k-l}.
\]
By the induction hypothesis, $z_k(t)\ge 0$ and Lemma \ref{linODE} implies that $\mu_k\ge 0.$ Thus,
\[
\|\{\mu_k\}\|_{l_1(\N)}\ =\ \sum_{k=1}^\infty\mu_k=1,
\]
so the solution to \eqref{eqq} is uniformly bounded on $[0,T_0)$,
and can therefore can be continued beyond $T_0$ (Dieudonn\'e (1960),
Theorem 10.5.6). Since $[0,T_0)$ is, by assumption, the maximal
existence interval, we have $T_0=+\infty.$
\end{proof}

We now expand the solution $\,\mu\,$ in a special manner. Namely, denote
\[
c_H\,-\,C_i\ =\ f_i\ \ge\ 0.
\]
We can then rewrite the equation as
\begin{equation}\label{eqqq}
\mu_k'\ =\ \eta\,\pi_k\ -\ (\eta\,+\,c_H^2)\,\mu_k\ +\ c_H\,f_k\,\mu_k\ +\ C_k\,\mu_k\,\sum_{i=1}^\infty\,f_i\,\mu_i +\ \sum_{l=1}^{k-1}\,C_l\mu_l\,C_{k-l}\mu_{k-l}.
\end{equation}
Now, we will (formally) expand
\[
\mu_k\ =\ \sum_{j=0}^\infty\,\mu_{k\,j}(t),
\]
where $\,\mu_{k\,0}\,$ does not depend on (the whole sequence) $\,(f_i)\,,\,$ $\,\mu_{k\,1}\,$ is linear in (the whole sequence) $\,(f_i)\,,\,$ $\,\mu_{k\,2}\,$ is quadratic in $\,(f_i)\,$, and so on. The main idea is to expand so that all terms of the expansion are nonnegative.

Later, we will prove that the expansion indeed converges and coincides with the unique solution to the evolution equation.

Substituting the expansion into the equation, we get
\begin{equation}\label{exp1}
\mu_{k\,0}'\ =\ \eta\,\pi_k\ -\ (\eta\,+\,c_H^2)\,\mu_{k\,0}\ +\ \sum_{l=1}^{k-1}\,C_l\mu_{l\,0}\,C_{k-l}\mu_{k-l\,,\,0},
\end{equation}
with the given initial conditions: $\mu_{k\,0}(0)\ =\ \mu_k(0)\,$ for all $\,k\,.$
Furthermore,
\begin{equation}\label{exp2}
\mu_{k\,1}'\ =\ -\,(\eta\,+\,c_H^2)\,\mu_{k\,1}\ +\ c_H\,f_k\,\mu_{k\,0}\ +\ C_k\,\mu_{k\,0}\,\sum_{i=1}^\infty\,f_i\,\mu_{i\,0}\ +\ 2\,\sum_{l=1}^{k-1}\,C_l\mu_{l\,0}\,C_{k-l}\mu_{k-l\,,\,1},
\end{equation}
and then
\begin{multline}\label{exp3}
\mu_{k\,j}'\ =\ -\,(\eta\,+\,c_H^2)\,\mu_{k\,j}\ +\ c_H\,f_k\,\mu_{k\,j-1}\ +\ 2\,C_k\,\sum_{m=0}^{j-1}\,\mu_{k\,m}\,
\sum_{i=1}^\infty\,f_i\,\mu_{i\,,\,j-1-m}\ \\
\!\!\!\!\!\!\!\!\!\!\!\!\!\!\!\!\!\!\!\!\!\!\!\!\!\!\!\!\!\!\!\!\!\!\!\!\!\!\!\!\!\!\!\!\!\!\!\!
+\ 2\,\sum_{l=1}^{k-1}\,\sum_{m=0}^{j-1}\,C_l\mu_{l\,m}\,C_{k-l}\mu_{k-l\,,\,j-m}.
\end{multline}
with initial conditions $\mu_{k\,j}(0)=0.$ Equations \eqref{exp2}-\eqref{exp3} are only well defined if $\mu_{i0}$ exists for all $t$ and the infinite series
\[
\sum_{i=1}^\infty\,\mu_{ij}(t)
\]
converges for all $t$ and all $j. $

Thus, we can solve these linear ODEs with the help of Lemma \ref{linODE}. This is done through a recursive procedure. Namely, the equation for $\,\mu_{1\,0}\,$ is linear and we have
\[
\mu_{1\,0}'\ =\ \eta\,\pi_1\ -\ (\eta+c_H^2)\,\mu_{1,0}\ \le\ \eta\,\pi_k\ -\ (\eta\,+\,c_H^2)\,\mu_{1\,0}\ +\ c_H\,f_1\,\mu_1\ +\ C_1\,\mu_1\,\sum_{i=1}^\infty\,f_i\,\mu_i.
\]
A comparison theorem for ODEs (Hartman (1982), Theorem 4.1, p. 26) immediately implies that $\mu_{1\,0}\ \le \mu_1$ for all $t.$ By definition, $\mu_{k\,0}$ solves
\[
\mu_{k\,0}'\ =\ \eta\,\pi_k\ -\ (\eta+c_H^2)\,\mu_{k\,0}\ +\ z_{k\,0},
\]
with
\[
z_{k\,0}\ =\ \sum_{l=1}^{k-1}\,C_l\mu_{l\,0}\,C_{k-l}\mu_{k-l\,,\,0}
\]
depending on only those $\,\mu_{l\,0}\,$ with $\,l\,<\,k\,.$ Since $\mu_{1\,0}$ is nonnegative, it follows by induction that all equations for $\mu_{k\,0}$ have nonnegative inhomogeneities and hence $\mu_{k\,0}$ is nonnegative for each $k.$ Suppose now that $\mu_{l\,0}\ \le\ \mu_l$ for all $l\le k-1.$ Then,
\[
z_{k\,0}\ =\ \sum_{l=1}^{k-1}\,C_l\mu_{l\,0}\,C_{k-l}\mu_{k-l\,,\,0}\ \le\ \sum_{l=1}^{k-1}\,C_l\mu_{l}\,C_{k-l}\mu_{k-l},
\]
and \eqref{eqqq} and the same comparison theorem imply $\mu_{k\,0}\ \le\ \mu_k.$ Thus, $\mu_{k\,0}\le \mu_k$ for all $k\,.$ It follows that the series
\[
\sum_k\mu_{k\,0}\ \le\ 1
\]
converges and therefore, equations \eqref{exp2} are well-defined. Let now
\[
\mu_k^{(N)}\ =\ \sum_{j=0}^N\,\mu_{k\,j}.
\]
Suppose that we have shown that $\mu_{k\,j}\ge 0$ for all $k$ and all $j\,\le N-1$ and that
\begin{equation}\label{indhyp}
\mu_{k}^{(N-1)}\ \le\ \mu_k
\end{equation}
for all $k.$ Equations \eqref{exp2}-\eqref{exp3} are again linear inhomogeneous and can be solved using Lemma \ref{linODE} and the nonnegativity of $\mu_{k\,N}$ follows.
Adding \eqref{exp1}-\eqref{exp3} and using the induction hypothesis \eqref{indhyp}, we get
\begin{multline}\label{dombo}
(\mu_k^{(N)})'\ \\ \le\ \eta\pi_k\ -\ (\eta+c_H^2)\mu_k^{(N)}\ +\ c_H\,f_k\,\mu_{k}^{(N)}\ +\ C_k\,\mu_{k}^{(N)}\sum_{i=1}^\infty\,f_i\,\mu_{i}\ \sum_{l=1}^{k-1}\,C_l\mu_{l}\,C_{k-l}\mu_{k-l}.
\end{multline}
The comparison theorem applied to \eqref{dombo} and \eqref{eqqq} implies that $\mu_{k}^{(N)}\ \le\ \mu_k.$ Thus, we have shown by induction that $\mu_{k\,j}\ \ge 0$ and
\[
\sum_{j=0}^N\,\mu_{k\,j}\ \le\ \mu_k
\]
for any $N\ge 0.$ The infinite series $\sum_{j=0}^\infty\,\mu_{k\,j}(t)\,$ consists of nonnegative terms and is uniformly bounded from above. Therefore, it converges to a function $\tilde\mu_k(t).$ Using Lemma \ref{lemconv} and adding up \eqref{exp1}-\eqref{exp3}, we get that the sequence $\{\tilde\mu_k(t)\}$ is continuously differentiable and satisfies \eqref{eqqq} and $\tilde\mu_k(0)=\mu_k(0).$ Since, by Proposition \ref{exun}, the solution to \eqref{eqqq} is unique, we get $\tilde\mu_k\ =\ \mu_k$ for all $k.$
Thus, we have proved

\begin{theorem}\label{limlim} We have
\[
\mu_k\ =\ \sum_{j=0}^\infty\,\mu_{k\,j}.
\]
\end{theorem}

It remains to prove that $\,\lim_{t\to\infty}\,\mu_k(t)\,$ exists. The strategy for this consists of two steps:
\begin{enumerate}
\item
 Prove that $\,\lim_{t\to\infty}\,\mu_{k\,j}(t)\ =\ \mu_{k\,j}(\infty)\,$ exists.

\item Prove that
\[
\lim_{t\to\infty}\,\sum_{j=0}^\infty\,\mu_{k\,j}\ =\ \sum_{j=0}^\infty\,\lim_{t\to\infty}\mu_{k\,j}.
\]
\end{enumerate}
Equation \eqref{exp1} and Lemma \ref{linODE} directly imply the convergence of $\,\mu_{k\,0}\,.$ But,
the next step is tricky because of the appearance of the infinite sums
\[
\sum_{i=0}^\infty\,f_i\,\mu_{i\,j}(t)
\]
in equations \eqref{exp2}-\eqref{exp3}. If we prove convergence of this infinite sum, a subsequent application of Lemma \ref{linODE} to \eqref{exp2}-\eqref{exp3} will imply convergence of $\mu_{i\,,\,j+1}. $
Unfortunately, convergence of $\,\mu_{i\,j}(t)\,$ for each $i,j$ is not enough for the convergence of the infinite sum.

Recall that, by assumption, there exists an $\,N\,$ such that $\,C_i\ =\ C_N\,$ for all $\,i\,\ge\,N\,$. Thus, we  need only show that
\[
M_j(t)\ =\ \sum_{i=N}^\infty\,\mu_{i\,j}(t)
\]
converges for each $\,j\,.$

We will start with the case $j=0.$ Then, adding up \eqref{exp1} and using Lemma \ref{lemconv}, we get
\[
M_0'(t)\,+\,\sum_{i=1}^{N-1}\,\mu_{i\,0}'\ =\ \eta\ -\ (\eta+c_H^2)\,\left(M_0(t)\,+\,\sum_{i=1}^{N-1}\,\mu_{i\,0}\right)\ +\ \left(\,\sum_{i=1}^{N-1}\,C_i\,\mu_{i\,0}\,+\,C_N\,M_0(t)\,\right)^2.
\]
Opening the brackets, we can rewrite this equation as a Riccati equation for $\,M_0:$
\[
M_0'(t)\ =\ a_0(t)\ +b_0(t)\,M_0(t)\ +\ C_N^2\,M_0(t)^2.
\]
A priori, we  know  that $\,M_0\,$ stays bounded and, by the above,  the
coefficients $\,a_0\,,\,b_0\,$ converge to finite limits.

\begin{lemma}\label{M0} $M_0(t)$ converges to a finite limit at $t\to\infty.$
\end{lemma}

To prove it, we will need an auxiliary

\begin{lemma}\label{aux1} Let $N(t)$ be the solution to
\[
N'(t)\ =\ a\ +\ bN(t)\ +\ c\,N^2(t)\ ,\ N(0)\ =\ N_0.
\]
If the characteristic polynomial $q(\gl)=a+b\gl+c\gl^2$ has real zeros $\gl_1\ge\gl_2$ then

(1) If $N_0<\gl_1$ then $\lim_{t\to\infty}N(t)\ =\ \gl_2$.

(2) If $N_0>\gl_1$ then $\lim_{t\to\infty}N(t)=+\infty.$

\noindent If $q(\gl)$ does not have real zeros, then $\lim_{t\to\infty}N(t)=+\infty$ for any $N_0.$
\end{lemma}

\begin{proof} The stationary solutions are $N=\gl_{1,2}.$ If $N_0<\gl_2$ then $N(t)< \gl_2$ for all $t$ by uniqueness. Hence, $N'(t)=a\ +\ bN(t)\ +\ c\,N^2(t)>0$ and $N(t)$ increases and converges to a limit $N(\infty)\le \gl_2.$ This limit should be a stationary solution, that is $N(\infty)=\gl_2.$ If $N_0\in(\gl_2,\gl_1)$ then $N(t)\in(\gl_2,\gl_1)$ for all $t$ by uniqueness and therefore $N'(t)<0$ and $N(t)$ decreases to $N(\infty)\ge \gl_2,$ and we again should have $N(\infty)=\gl_2.$ If $N_0>\gl_1,$ $N'>0$ and hence
\[
N'(t)\ =\ a\ +\ bN(t)\ +\ c\,N^2(t)\ >\ a\ +\ bN_0\ +\ c\,N^2_0\ >\ 0
\]
for all $t$ and the claim follows. If $q(\gl)$ has no real zeros, its minimum $\min_{\gl\in\R}q(\gl)\ =\ \gd$ is strictly positive. Hence, $N'(t)>\gd>0$ and the claim follows.
\end{proof}

\bigskip

\begin{proof}[Proof of Lemma \ref{M0}] Consider the quadratic polynomial
\[
q_\infty(\gl)\ =\ a_0(\infty)+b_0(\infty)\gl+C_N^2\gl^2\ =\ 0.
\]
We will consider three cases:
\begin{enumerate}

 \item $q_\infty(\gl)$ does not have real zeros, that is, $\min_{\gl\in\R}\,q_\infty(\gl)\,=\,\gd>\,0.$ Then,  for all sufficiently large $t,$
\begin{equation}\label{lobo}
M_0'(t)\ =\ a_0(t)\ +b_0(t)\,M_0(t)\ +\ C_N^2\,M_0(t)^2\ \ge\ \gd/2>0,
\end{equation}
so $M_0(t)$ will converge to $+\infty,$ which is impossible.

 \item $q_\infty(\gl)$ has a double zero $\gl_\infty.$ Then, we claim that $\lim_{t\to\infty}M(t)=\gl_\infty.$ Indeed, suppose it is not true. Then, there exists an $\eps>0$ such that either $\sup_{t>T}(M(t)-\gl_\infty)>\eps$ for any $T>0$ or $\sup_{t>T}(\gl_\infty-M(t))>\eps$ for any $T>0.$ Suppose that the first case takes place. Pick a $\gd>0$ and choose $T>0$ so large that
\[
a_0(t)\ge a_0(\infty)-\gd,\quad \quad  b_0(t)\ \ge\ b_0(\infty)-\gd
\]
for all $t\ge T.$ The quadratic polynomial
\[
a_0(\infty)-\gd\ +(b_0(\infty)-\gd)\,\gl\ +\ C_N^2\,\gl^2
\]
has two real zeros $\gl_{1}(\gd)>\gl_2(\gd)$ and, for sufficiently small $\gd,$ $|\gl_{1,2}(\gd)-\gl_\infty|<\eps/2.$
Let $T_0>T$ be such that $M_0(T_0)>\gl_\infty+\eps.$ Consider the solution $N(t)$ to
\[
N'(t)\ =\ a_0(\infty)-\gd\ +(b_0(\infty)-\gd)\,N(t)\ +\ C_N^2\,N(t)^2\ ,\ N(T_0)=M(T_0)\ >\ \gl_1(\gd).
\]
By the comparison theorem for ODEs, $M_0(t)\ \ge\ N(t)$ for all $t\ge T_0$ and Lemma \ref{aux1} implies that $N(t)\to\infty,$ which is impossible.
Suppose now $\sup_{t>T}(\gl_\infty-M(t))>\eps$ for any $T>0.$ Consider the same $N(t)$ as above and choose $\gd$ so small that $M(T_0)<\gl_2(\gd).$ Then, $M(t)\ \ge\ N(t)$ and, by Lemma \ref{aux1}, $N(t)\to\gl_2(\gd).$ For sufficiently small $\gd,$ $\gl_2(\gd)$ can be made arbitrarily close to $\gd_\infty$ and hence  $\sup_{t>T}(\gl_\infty-M(t))>\eps$ cannot hold for sufficiently large $T,$ which is a contradiction.

 \item $q(\gl)$ has two distinct real zeros $\gl_1(\infty)>\gl_2(\infty).$ Then, we claim that either $\lim_{t\to\infty}M(t)=\gl_1(\infty)$ or $\lim_{t\to\infty}M(t)=\gl_2(\infty).$ Suppose the contrary. Then, there exists an $\eps>0$ such that $\sup_{t>T}|M(t)-\gl_1(\infty)|>\eps$ and $\sup_{t>T}|M(t)-\gl_2(\infty)|>\eps.$ Now, an argument completely analogous to that of case (2) applies.
\end{enumerate}
\end{proof}

\bigskip

With this result, we go directly to \eqref{exp2} and get the convergence of $\mu_{j\,1}\,$
from Lemma \ref{linODE}. Then, adding equations \eqref{exp2}, we get a linear equation for $\,M_1(t)\,$ and again get convergence from Lemma \ref{linODE}. Note that we are in an even simpler situation of linear equations. Proceeding inductively, we arrive at

\begin{prop} The limit
\[
\lim_{t\to\infty}\,\mu_{k\,j}(t)\ \le\ 1
\]
exists for any $k,j\,.$
\end{prop}

The next important observation is: If we add an $\,\eps\,$ to $\,c\,$ in the second quadratic term, and subtract $\,\eps\,$ from $c$ in the first quadratic term, then the measure remains bounded. This is a non-trivial issue, since the derivative could become large.

\begin{lemma}\label{dominant} If $\,\eps\,$ is sufficiently small, then  there exists a constant $K>\,0\,$ such that the system
\[
\mu_k'\ =\ \eta\,\pi_k\ -\ \eta\,\mu_k\ -\ (C_k-\eps)\,\mu_k\,\sum_{i=1}^\infty\,(C_i-\eps)\,\mu_i\ +\ \sum_{l=1}^{k-1}\,C_l\,\mu_l\,C_{k-l}\,\mu_{k-l}
\]
has a unique solution $\{\mu_k(t)\}\in l_1(\N)$ and this solution satisfies
\[
\sum_{k=1}^\infty\,\mu_k(t)\ \le\ K
\]
for all $\,t\,>0\,.$
\end{lemma}

\begin{proof} An argument completely analogous to that in the proof of Proposition \ref{exun} implies that the solution exists on an interval $[0,T_0),$ is unique and nonnegative.
Letting
\[
M\ =\ \sum_{k=1}^\infty\,\mu_k(t)
\]
and adding up the equations, we get
\[
M'\ \le\ \eta\ -\ \eta\,M\ +\,2\,\eps\,c_H\,M^2\ .
\]
Consider now the solution $\,N\,$ to the equation
\begin{equation}\label{N}
N'\ =\ \eta\ -\ \eta\,N\ +\ 2\eps\,c_H\,N^2.
\end{equation}
By a standard comparison theorem for ODEs (Hartman (1982), Theorem 4.1, p. 26), if $N(0)\ =\ M(0),$ then $\,M(t)\,\le\,N(t)\,$ for all $\,t\,.$ Thus, we only need to show that $N(t)$ stays bounded.
The stationary points of \eqref{N} are
\[
d_{1,2}\ =\ \frac{\eta\ \pm\ \sqrt{\eta^2\,-\,8\eps c_H\,\eta\, }}{4\eps c_H},
\]
so, for sufficiently small $\,\eps\,,\,$ the larger stationary point $\,d_1\,$ is arbitrarily large. Therefore, by uniqueness, if $\,N(0)\,<\,d_1\,,\,$ then $\,N(t)\,<\,d_1$ for all $t$, and we are done. Now, the same argument as in the proof of Proposition \ref{exun} implies that $T_0=+\infty$ and the proof is complete.
\end{proof}

Since $\,(f_i)\,$ is bounded, for any $\,\eps\,>\,0$ there exists a $\,\gd\,>\,0\,$ such that
\[
f_i\,+\,\eps\ \ge\  (1+\gd)\,f_i
\]
for all $\,i\,.$ Consider now the solution $\,(\mu_k^{(\eps)})\,$ to the equation of Lemma \ref{dominant}. Then, using the same expansion
\[
\mu_k^{(\eps)}\ =\ \sum_{j=0}^\infty\,\mu_{k\,j}^{(\eps)},
\]
we immediately get (by direct calculation) that
\[
(1+\gd)^j\,\mu_{k\,j}\ \le\ \mu_{k\,j}^{(\eps)}.
\]
By Lemma \ref{dominant},
\[
\mu_{k\,j}^{(\eps)}\ \le\ \sum_{k,j}\,\mu_{k\,j}^{(\eps)}\ =\ \sum_k\,\mu_k^{(\eps)}\ <\ K,
\]
and therefore
\[
\mu_{k\,j}(t)\ \le\ \frac{K}{(1+\gd)^j},
\]
for some (possibly very large) constant $K$, independent of $\,t\,.$ Now we are ready to prove

\begin{theorem} We have
\[
\lim_{t\to\infty}\,\mu_k(t)\ =\ \sum_{j=0}^\infty\,\mu_{k\,j}(\infty).
\]
\end{theorem}

\begin{proof} Take $\,N\,$ so large that
\[
\sum_{j=N}^\infty\,\mu_{k\,j}(t)\ \le\ \sum_{j=N}^\infty\,\frac{K}{(1+\gd)^j}\ <\ \eps/2
\]
for all $\,t\,.$ Then, choose $\,T\,$ so large that
\[
\sum_{j=0}^{N-1}\,|\mu_{k\,j}(t)-\mu_{k\,j}(\infty)|\ <\ \eps/2
\]
for all $\,t\,>\,T\,.$ Then,
\[
\left|\,\sum_{j=0}^{\infty}\,(\mu_{k\,j}(t)-\mu_{k\,j}(\infty))\,\right|\ <\ \eps.
\]
Since $\,\eps\,$ is arbitrary, we are done.
\end{proof}

\begin{lemma}
Consider the limit
\[
\bar C\ =\ \lim_{t\to\infty}\,\sum_n\,C_n\,\mu_n(t).\
\]
In general,
\[
\bar C \ge\  \sum_n C_n\lim_{t\to\infty}\mu_n(t)\ =\ \tilde C,
\]
with equality  if and only if
\[
\sum_n\,\lim_{t\to\infty}\mu_n(t)\ =\ 1.
\]
Furthermore,
\[
\bar C\ -\ \tilde C\ =\ C_N\,\left(1\ -\ \sum_n\,\lim_{t\to\infty}\mu_n(t)\right).
\]
\end{lemma}

Based on this lemma, we have

\begin{lemma} The limit $\mu_\infty(n)\ =\ \lim_{t\to\infty}\mu_t(n)\,$ satisfies the equation
\[
0\ =\ \eta(\pi\ -\ \mu_\infty)\ +\ \mu_\infty^C*\mu_\infty^C\ -\ \mu_\infty^C\,\bar C,
\]
where, in general,
\[
\bar C\ \ge\ \mu^C_\infty(\N)\
\]
and
\[
\mu_\infty(\N)\ =\ 1\ -\ \eta^{-1}\,\mu^C_\infty\,(\bar C\,-\,\mu_\infty^C\,)\ \le\ 1.
\]
\end{lemma}

An immediate consequence is

\begin{lemma} The limit distribution $\mu_\infty$ is a probability measure and coincides with the unique solution to \eqref{stationary} if and only if
$
\bar C\ =\ \tilde C.$
\end{lemma}

\begin{prop} \label{nolostmass}
Under the same tail condition  $C_n=C_N$ for $n\geq N$ and the condition that
\begin{equation}\label{noloss}
\eta\ \ge\ C_N\,c_H,
\end{equation}
we have
\[
\bar C\ =\ \tilde C,
\]
and, therefore, $\mu_\infty$ is a probability measure that coincides with the unique solution to \eqref{stationary}.
\end{prop}

\medskip

\begin{proof}
Recall that the equation for the limit  measure is
\begin{equation}\label{ctcb}
-\eta \mu_\infty(k)\ +\ \eta\pi_k - \bar C C_k \mu_\infty(k) + \sum_l C_l\mu_\infty(l)\,C_{k-l} \mu_\infty({k-l}) = 0,
 \end{equation}
 where
 \[
 \bar C\ =\ \lim_{t\to\infty}\,\sum_n\,C_n\,\mu_n(t)\ \ge \sum_n C_n\mu_\infty(n)\ =\ \tilde C.
 \]
The difference,
 \[
 \bar C\ -\ \tilde C\ =\ (1-M)\,C_N,
 \]
 with
 \[
M\ =\ \sum_k\mu_\infty(k),\
\]
is non-zero if and only if there is a loss of mass, that is $M<1.$

Adding \eqref{ctcb} up over $\,k\,,\,$ we get
 \begin{multline}
-\eta M\ +\ \eta\ -\bar C\tilde C\ +\ (\tilde C)^2\ =\ 0\ \\ =\ -\eta\,M\,+\,\eta\,-\,\bar C\,(\bar C\,-\,(1-M)\,C_{N})\,+\,(\bar C\,-\,(1-M)\,C_{N})^2.
 \end{multline}
If $M\not=1$, we get, dividing this equation by $\,(1-M)\,,\,$ that
 \[
 M\ =\ 1\ +\ \frac{\eta\ -\ C_{N}\,\bar C}{C_{N}^2}.
 \]
Since $M\,\le\,1\,,\,$ we immediately get that if
\[
\eta\ \ge\ C_{N}\,c_H,
\]
then there is no loss of mass, proving the result.
\end{proof}

\bigskip

\subsection{Monotonicity properties of the stationary measure.}

We recall that the equation for the stationary measure
$
\mu\ =\ (\mu_k\,,\,k\,\ge\,1)
$
is
\[
-\,\eta\,\mu_k\ +\ \eta\,\pi_k\ -\ C_k\,\mu_k\,\sum_{i=1}^\infty\,C_i\,\mu_i \ +\ \sum_{l=1}^{k-1}\,C_l\,C_{k-l}\,\mu_l\,\mu_{k-l}\ =\ 0\,.
\]

We write $\{\bar\mu_k(\overline C, C_1,\ldots,C_k): k\geq 1\}$ for the measure constructed
recursively in the statement of Lemma \ref{stat_mes} from a given
$\overline C$ and a given policy $C$.

\begin{lemma}\label{monoton1} For each $\,k\,,\,$ the function that maps $C$ to
\[
C_k\bar\mu_k\ =\ C_k\,\bar\mu_k(\bar C\,,\,C_1\,,\,\ldots\,,\,C_k)
\]
is monotone increasing in $\,C_i\,,\,i\,=\,1\,,\ldots\,,\,k\,$, and monotone decreasing in $\,\bar C\,.$
\end{lemma}

\begin{proof} The proof is by induction. For $k=1,$ there is nothing to prove. For $k>1$,
\[
C_k\,\bar\mu_k\ =\ \frac{C_k}{\eta\,+\,\bar C\,C_k}\,\left(\eta\,\pi_k\ +\ \sum_{l=1}^{k-1}\,\big(C_l\,\bar\mu_l)\,\big(C_{k-l}\,\bar\mu_{k-l})\,\right)
\]
is monotone, by the induction hypothesis.
\end{proof}

\bigskip\noindent{\bf Steps toward a proof of Proposition \ref{M>N}}

Our proof of Proposition \ref{M>N} is based on
the next series of results, Lemmas \ref{Zi} through \ref{min_dom},
 for which we use the notation
\[
\nu_k\ =\nu_k(C_1\,,\,\ldots\,,\,C_k\,,\,\bar C)=\ C_k\,\bar\mu_k(\bar C, C_1,\ldots,C_k)
\]
and
\[
Z_k\ =\ \frac{C_k}{\eta\ +\ C_k\,\bar C}.
\]
Note that, multiplying $\,\sqrt{\eta}\,$ and $\,C_k\,$ by the same number $\,\gl\,$ does not change the equation, and hence
does not change the stationary measure. Thus, without loss of generality, we can normalize for simplicity to the case $\,\eta\,=\,1\,$.

\begin{lemma}\label{Zi} The sequence $\,\nu_k\,$ satisfies
\[
\nu_1\ =\ Z_1\pi_1,
\]
\[
\nu_k\  =\ Z_k\,\left(\pi_k\ +\ \sum_{l=1}^{k-1}\,\nu_l\,\nu_{k-l}\right),
\]
and
\begin{equation}\label{nuic}
\sum_{k=1}^\infty\,\nu_k(C_1\,,\,\ldots\,,\,C_k\,,\,\bar C)\ =\ \bar C.
\end{equation}
\end{lemma}

Differentiating \eqref{nuic} with respect to $C_k,$ we get

\begin{lemma}
\[
\frac{\partial \bar C}{\partial C_k}\ =\ \frac{\sum_{i=k}^\infty\,\frac{\partial \nu_i}{\partial C_k}}{1\ -\ \sum_{i=1}^\infty\,\frac{\partial \nu_i}{\partial \bar C}}.
\]
\end{lemma}

\bigskip
We now let $\bar C(C)$ denote the unique solution of the equation
(\ref{nuic}),
and for any policy $C$ we define
\[
\xi_k(C)\ =\ \nu_k(C_1,\ldots,C_k\,,\,\bar C(C)).
\]

\begin{lemma}\label{mainlem} Let $\,m\,<\,k\,.$ We have
\[
\frac{\partial}{\partial C_m}\,\sum_{i=k}^\infty\,\xi_i\ \ge\ 0
\]
if and only if
\begin{equation}
\left(1\,-\,\sum_{i=1}^{k-1}\,\frac{\partial \nu_i}{\partial \bar C}\,\right)\,\sum_{i=k}^\infty\,C_m^2\frac{\partial \nu_i}{\partial C_m}\ \ge\ \left(\sum_{i=1}^{k-1}\,C_m^2\,\frac{\partial \nu_i}{\partial C_m}\right)\,\sum_{i=k}^\infty\,\left(-\frac{\partial \nu_i}{\partial \bar C}\,\right).
\end{equation}
\end{lemma}

\begin{proof} We have
\begin{eqnarray*}
\frac{\partial}{\partial C_m}\,\sum_{i=k}^\infty\,\xi_i &=&\ \sum_{i=k}^\infty\, \left(\frac{\partial\nu_i}{\partial C_m}\,+\, \frac{\partial\nu_i}{\partial \bar C}\,\frac{\partial \bar C}{\partial C_m}\right)\ \\
&=&\
\frac{
\left(1\ -\ \sum_{i=1}^\infty\,\frac{\partial \nu_i}{\partial \bar C}\right)\,\sum_{i=k}^\infty\,\frac{\partial\nu_i}{\partial C_m}\,+\, \sum_{i=k}^\infty\frac{\partial\nu_i}{\partial \bar C}\,\left(\sum_{i=m}^{k-1}\,+\,\sum_{i=k}^\infty\right)\,\frac{\partial \nu_i}{\partial C_m}}{1\ -\ \sum_{i=1}^\infty\,\frac{\partial \nu_i}{\partial \bar C}}\\
&=&\ \frac{\left(1\ -\ \sum_{i=1}^{k-1}\,\frac{\partial \nu_i}{\partial \bar C}\right)\,\sum_{i=k}^\infty\,\frac{\partial\nu_i}{\partial C_m}\,+\, \sum_{i=k}^\infty\frac{\partial\nu_i}{\partial \bar C}\,\sum_{i=m}^{k-1}\,\frac{\partial \nu_i}{\partial C_m}}{1\ -\ \sum_{i=1}^\infty\,\frac{\partial \nu_i}{\partial \bar C}},\
\end{eqnarray*}
and the claim follows.
\end{proof}

Suppose now that we have the ``flat-tail'' condition that for some $N$, $C_n\ =\ C_N\,$ for all $\,n\,\ge\,N\,.$ We define
the moment-generating function $m(\,\cdot\,)$ of $\nu$ by
\begin{equation}\label{m(x)}
m(x)\ =\ \sum_{i=1}^\infty\,\nu_i\,x^i.
\end{equation}
By definition, the first $\,N-1\,$ coefficients $\,\nu_i\,$ of the power series expansion of $m(x)$  satisfy
\[
\nu_1\ =\ \pi_1\,Z_1\,
\]
and then
\[
\nu_i\ =\ \left(\pi_i\ +\ \sum_{l=1}^{i-1}\,\nu_l\,\nu_{i-l}\,\right)\,Z_i
\]
for $\,i\,\le\,N-1\,$, and
\[
\nu_i\ =\ Z_N\,\left(\pi_i\ +\ \sum_{l=1}^{i-1}\,\nu_l\,\nu_{i-l}\,\right)
\]
for $\,i\,\ge\,N\,.$
For $\,N\,\ge\,2\,$, let
\[
m^b(x,N)\ =\ \sum_{i=N}^\infty\,\pi_i\,x^i.\
\]
Using Lemma \ref{Zi} and comparing the coefficients in the powers-series expansions, we get

\begin{lemma} If $C_n\ =\ C_N\,$ for all $\,n\,\ge\,N\,$, then
\[
m(x)\ -\ \sum_{i=1}^{N-1}\,\nu_i\,x^i\ =\ Z_N\,\left(m^b(x,N)\ +\ m^2(x)\ -\ \sum_{i=2}^{N-1}\,x^i\,\sum_{l=1}^{i-1}\,\nu_l\,\nu_{i-l}\,\right).
\]
\end{lemma}

\bigskip
Thus,
\begin{eqnarray*}
0\ &=&\ Z_N\,\left(m^b(x,N)\ +\ m^2(x)\ -\ \sum_{i=2}^{N-1}\,x^i\,\sum_{l=1}^{i-1}\,\nu_l\,\nu_{i-l}\,\right)\ -\ m(x)\ +\  \sum_{i=1}^{N-1}\,\nu_i\,x^i\ \\
&=&\,
Z_N\,\left(m^b(x,N)\ +\ m^2(x)\ -\ \sum_{i=2}^{N-1}\,x^i\,\sum_{l=1}^{i-1}\,\nu_l\,\nu_{i-l}\,\right)\ \\
&&\qquad\qquad\qquad-\ m(x)\ +\  \pi_1\,Z_1\,x\ +\ \sum_{i=2}^{N-1}\,Z_i\,\left(\pi_i\ +\ \sum_{l=1}^{i-1}\,\nu_l\,\nu_{i-l}\,\right)\,x^i\ \\
&=& \,
Z_N\,m^2(x)\ -\ m(x)\ +\ Z_N\,\Bigg(m^b(x,N)\ +\ \sum_{i=2}^{N-1}\,\pi_i\,x^i\
-\  \sum_{i=2}^{N-1}\,\left(\pi_i\ +\ \sum_{l=1}^{i-1}\,\nu_l\,\nu_{i-l}\,\right)\,\,x^i\Bigg)\\
&& \quad\quad +\ \pi_1\,Z_1\,x\ +\ \sum_{i=2}^{N-1}\,Z_i\,\left(\pi_i\ +\ \sum_{l=1}^{i-1}\,\nu_l\,\nu_{i-l}\,\right)\,x^i\ \\
&& \quad \\
&=& \ Z_N\,m^2(x)\ -\ m(x)\ +\ Z_N\,m^b(x,2)\ +\ \pi_1\,Z_1\,x\,\\
&& \qquad\qquad\qquad+\,\sum_{i=2}^{N-1}\,x^i\,(Z_i\,-\,Z_N)\,\left(\pi_i\,+\,\sum_{l=1}^{i-1}\,\nu_l\,\nu_{i-l}\,\right).
\end{eqnarray*}
Solving this quadratic equation  for $m(x)$ and picking the branch that satisfies $m(0)=0,$ we arrive at
\begin{lemma}\label{lem(x)} The moment generating function $m(\,\cdot\,)$ of $\nu$ is given by
\begin{equation}\label{eqmx}
m(x)\ =\ \frac{1\ -\ \sqrt{1-4Z_NM(x)}}{2\,Z_N},
\end{equation}
where
\[
M(x)\ =\ Z_N\,m^b(x,2)\ +\ \pi_1\,Z_1\,x\,+\,\sum_{i=2}^{N-1}\,x^i\,(Z_i\,-\,Z_N)\,\left(\pi_i\,+\,\sum_{l=1}^{i-1}\,\nu_l\,\nu_{i-l}\,\right).\,
\]
\end{lemma}

\bigskip
The derivatives of the functions $\,Z_k\,$  satisfy interesting algebraic identities, summarized in
\begin{lemma}\label{derZ} We have $\,Z_k\ <\ 1$ for all $k$. Moreover,
\[
C_k^2\,\frac{\partial Z_k}{\partial C_k}\ =\ -\,\frac{\partial Z_k}{\partial \bar C}\ =\ Z_k^2.
\]
\end{lemma}

\bigskip
These identities allow us to calculate derivatives is an elegant form.
Let
\[
\gamma(x)\ =\ (1\ -\ 4Z_N\,\zeta(x))^{-1/2} ,
\]
with
\begin{equation}\label{zeta}
\zeta(x)\ =\ Z_N\,m^b(x,2)\ +\ \pi_1\,Z_1\,x\,+\,\sum_{i=2}^{N-1}\,x^i\,(Z_i\,-\,Z_N)\,\left(\pi_i\,+\,\sum_{l=1}^{i-1}\,\nu_l\,\nu_{i-l}\,\right).\,
\end{equation}
Differentiating identity \eqref{eqmx}  with respect to $\,\bar C$ and $\,C_{N-1}\,$ and using Lemma \ref{derZ}, we arrive at

\begin{lemma}
We have
\begin{multline} \label{derbarc}
-\frac{\partial m(x)}{\partial \bar C}\
=\ -\frac12\ +\ \gamma(x)\,
\Bigg(\frac12\,+\,\sum_{i=1}^{N-1}\,x^i\,Z_i\,(Z_i\,-\,Z_N)\,\left(\pi_i\,+\,\sum_{l=1}^{i-1}\,\nu_l\,\nu_{i-l}\,\right)\, \\+\
\sum_{i=2}^{N-1}\,x^i\,(Z_i\,-\,Z_N)\,\frac{-\partial}{\partial \bar C}\sum_{l=1}^{i-1}\,\nu_l\,\nu_{i-l}\,\Bigg)
\end{multline}
and
\begin{equation}
\frac{\partial m(x)}{\partial C_{N-1}}\ =\
\gamma(x)\,x^{N-1}\,Z_{N-1}^2\,\Bigg(\pi_{N-1}\,+\,\sum_{l=1}^{N-2}\,\nu_l\,\nu_{N-1-l}\Bigg).
\end{equation}
\end{lemma}

Let now
\begin{equation}
\gamma(x)\ =\ \sum_{j=0}^\infty\,\gamma_j\,x^j,
\end{equation}
and let $\,\gamma_j\ =\ 0$ for $\,j\,<\,0\,.$

\begin{lemma} We have
$
\gamma_j\ \ge\ 0
$
for all $j\,.$
\end{lemma}

\begin{proof} By \eqref{zeta}, the function $\zeta(\,\cdot\,)$ has nonnegative Taylor coefficients.
Thus, it suffices to show that the function that maps $x$ to
$
(1-x)^{-1/2}
$
has nonnegative Taylor coefficients. We have
\[
(1-x)^{-1/2}\ =\ 1\ +\ \sum_{k=1}^\infty\,x^k\,\beta_k,
\]
with
\begin{eqnarray*}
\beta_k\ &=&\ (-1)^k\,\frac{(-0.5) (-0.5-1)(-0.5-2)\cdots (-0.5-k+1)}{k!}\ \\
&& \\
&=&\ \frac{(0.5) (0.5+1)(0.5+2)\cdots (0.5+k-1)}{k!}\ >\ 0.
\end{eqnarray*}
Therefore,
\[
\gamma(x)\ =\ 1\ +\ \sum_{k=1}^\infty\,\beta_k\,(4Z_N\zeta(x))^k
\]
also has nonnegative Taylor coefficients.
\end{proof}

Let
\[
Q_{N-1}\ =\ \sum_{l=1}^{N-2}\,\nu_l\,\nu_{N-1-l}\,+\,\pi_{N-1}.
\]
Define also
\[
R_0\ =\ \frac12\ ,\ R_1\ =\ Z_1\,(Z_1-Z_N)\,\pi_1,
\]
and
\[
R_i\ =\ (Z_i\,-\,Z_N)\,\left(\nu_i\, +\ \frac{-\partial}{\partial \bar C}\sum_{l=1}^{i-1}\,\nu_l\,\nu_{i-l}\,\right).
\]
Recall that $Z_i\ >\ Z_N\,$ if and only if $C_i\ >\ C_N$ and therefore, $R_i\,\ge\,0\,$ for all $i$ as soon as $C_i\,\ge\,C_N$ for all $i\,\le\,N-1\,.$

\begin{lemma} \label{dernu} We have
\[
\frac{\partial \nu_j}{\partial C_{N-1}}\ =\ Z_{N-1}^2\,Q_{N-1}\,\gamma_{j-N+1}
\]
and
\[
-\frac{\partial \nu_j}{\partial \bar C}\ =\ \sum_{i=0}^{N-1}\,R_i\,\gamma_{j-i}\ =\
Z_{N-1}^{-2}\,Q_{N-1}^{-1}\sum_{i=0}^{N-1}\,R_i\,\frac{\partial \nu_{j-i+N-1}}{\partial C_{N-1}}.
\]
\end{lemma}

\begin{proof} Identity \eqref{derbarc} implies that
\[
\frac{-\partial m(x)}{\partial \bar C}\ =\ \sum_{j=1}^\infty\,x^j\,\sum_{i=0}^{N-1}\,R_i\,\gamma_{j-i}.
\]
On the other hand, by \eqref{m(x)},
\[
\frac{-\partial m(x)}{\partial \bar C}\ =\ \sum_{j=1}^\infty\,x^j\,\frac{-\partial\nu_j}{\partial \bar C}.
\]
Comparing Taylor coefficients in the above identities, we get the required result. The case of
 the derivative $\partial/\partial C_{N-1}$ is analogous.
\end{proof}

\begin{lemma} We have
\[
(R_0\,+\,\cdots\,+\,R_{N-1})\,Z_{N-1}^{-2}\,Q_{N-1}^{-1}\,\sum_{j=k}^\infty\,\frac{\partial \nu_j}{\partial C_{N-1}}\ \ge\
\sum_{j=k}^\infty\,\frac{-\partial \nu_j}{\partial \bar C}
\]
and
\[
R_0\,\gamma_0\ +\ \sum_{i=1}^{k-1}\,\frac{-\partial \nu_j}{\partial \bar C}\ \ge\ (R_0\,+\,\cdots\,+\,R_{N-1})\,Z_{N-1}^{-2}\,Q_{N-1}^{-1}\,\sum_{j=1}^{k-1}\,\frac{\partial \nu_j}{\partial C_{N-1}}.
\]
\end{lemma}

\begin{proof} By Lemma \ref{dernu},
\begin{eqnarray*}
\sum_{j=k}^\infty\,\frac{-\partial \nu_j}{\partial \bar C}\ & =&
\ \sum_{j=k}^\infty\, \sum_{i=0}^{N-1}\,R_i\,\gamma_{j-i}\
\\
&=&
\ \sum_{i=0}^{N-1}\,R_i\,\sum_{j=k-i}^\infty\,\gamma_j\ \\
& \le &
\ \sum_{i=0}^{N-1}\,R_i\,\sum_{j=k-N+1}^\infty\,\gamma_j\  \\
&= &\ (R_0\,+\,\cdots\,+\,R_{N-1})\,Z_{N-1}^{-2}\,Q_{N-1}^{-1}\,\sum_{j=k}^\infty\,\frac{\partial \nu_j}{\partial C_{N-1}}
\end{eqnarray*}
and
\begin{eqnarray*}
R_0\,\gamma_0\ +\ \sum_{j=1}^{k-1}\,\frac{-\partial \nu_j}{\partial \bar C}\
&=& \
R_0\gamma_0\ +\ \sum_{j=1}^{k-1}\, \sum_{i=0}^{N-1}\,R_i\,\gamma_{j-i}\ \\
&=& \ R_0\,\sum_{i=0}^{k-1}\gamma_j\ +\ \sum_{i=1}^{N-1}\,R_i\,\sum_{j=0}^{k-1-i}\gamma_i\ \\
& \ge & \ (R_0\,+\,\cdots\,+\,R_{N-1})\,\sum_{j=0}^{k-N+1}\,\gamma_j\ \\
&=& \ (R_0\,+\,\cdots\,+\,R_{N-1})\,Z_{N-1}^{-2}\,Q_{N-1}^{-1}\,\sum_{j=1}^{k-1}\,\frac{\partial \nu_j}{\partial C_{N-1}}.
\end{eqnarray*}
\end{proof}

By definition, $\,R_0=1/2$ and $\,\gamma_0\ =\ 1\,.$
Hence, we get

\begin{lemma}\label{fosdpr} Suppose that $C_i\,\ge\,C_N\,$ for all $i\,\le\,N$ and $C_i=C_N$ for all $i\,\ge\,N\,.$ Then, for all $\,m\,\ge\,N\,$,
\[
\left(1\ -\ \sum_{i=1}^{m-1}\,\frac{\partial \nu_i}{\partial \bar C}\,\right)\,\sum_{k=m}^\infty\,\frac{\partial \nu_i}{\partial C_{N-1}}\ \ge\ \sum_{i=N-1}^{m-1}\,\frac{\partial \nu_i}{\partial C_{N-1}}\,\,\sum_{k=m}^\infty\,\frac{-\partial \nu_i}{\partial \bar C}.
\]
\end{lemma}

\begin{lemma}\label{min_dom}  The function $\lambda$ defined by
\[
\lambda_n((C_i))\ =\ \sum_{i=n}^\infty\,\xi_i((C_i))
\]
is monotone increasing in $C_k\,$ for all $\,k\,\ge\,n\,.$
\end{lemma}

\begin{proof} By Proposition \ref{barc},
\[
\lambda_1\ =\ \bar C\ =\ \sum_{i=1}^\infty\,\xi_i
\]
is monotone increasing in $C_k\,$ for each $k\,.$ By Lemma \ref{stat_mes} and Proposition \ref{barc},
$\xi_i\ =\ \bar\mu_i(C_1,\ldots,C_{i-1},\bar C)\,$ is monotone decreasing in $C_k\,$ for all $k\ge i\,.$ Thus,
\[
\gl_n\ =\ \bar C\ -\ \sum_{i=1}^{n-1}\,\xi_i
\]
is monotone increasing in $C_k\,.$
\end{proof}

We are now ready to prove Proposition \ref{M>N}.

\bigskip

\begin{proof}[Proof of Proposition \ref{M>N}] It suffices to prove the claim for the case $N\ =\ M+1\,.$ It is known that $\,\mu^{C,N}\,$ dominates $\,\mu^{C,M}\,$ is the sense of first order stochastic dominance if and only if
\begin{equation}\label{FOSD1}
\sum_{i=k}^\infty\,\mu^{C,N}_i\ \ge\ \sum_{i=k}^\infty\,\mu^{C,M}_i
\end{equation}
for any $\,k\,\ge\,1\,.$ The only difference between policies $C^M$ and $C^{M+1}$ is that
\[
C^{M+1}_{M}\ =\ c_H\ >\ c_L\ =\ C^M_{M}.
\]
By Lemma \ref{min_dom}, \eqref{FOSD1} holds for any $\,k\,\le\,M\,$. By Lemmas \ref{mainlem} and \ref{fosdpr}, this is also true for $k\,>M\,.$ The proof of  Proposition \ref{M>N} is complete.
\end{proof}

\bigskip

\begin{proof}[Proof of Proposition \ref{barc}]
By Lemma \ref{monoton1}, the function $f$ that maps $(\bar C, C)$ to
\[
f(\bar C\,,\,(C_i\,,\,i\,\ge 1))\ =\ \sum_{k=1}^\infty\,C_k\,\bar\mu_k(\bar C,C_1,\ldots\,,\,C_k)
\]
is monotone increasing in $\,C_i\,$ and decreasing in $\,\bar C\,.$ Therefore, given $C$, the unique solution $\bar C$
to the equation
\begin{equation}
\label{BARCSOL}
\bar C\ -\ f(\bar C\,,\,(C_i\,,\,i\,\ge\,1))\ =\ 0
\end{equation}
is monotone increasing in $C_i$ for all $i$.
\end{proof}

\bigskip\noindent
{\bf Proof of Example \ref{counterexample} of non-monotonicity}

Let $\,C_n=0$ for $\,n\,\ge\,3\,$, as stipulated. We will check the condition
\[
\frac{\partial \nu_2}{\partial C_1}\ >\ 0.
\]
By the above, we need to check that
\[
\left(1\,-\,\frac{\partial \nu_1}{\partial \bar C}\right)\frac{\partial \nu_2}{\partial C_1}\ >\
\frac{\partial \nu_1}{\partial C_1}\,\frac{-\partial \nu_j}{\partial \bar C}.
\]
We have
\[
\nu_1\ =\ \pi_1\,Z_1
\]
and
\[
\nu_2\ =\ (\pi_2\,+\,(\pi_1\,Z_1)^2)\,Z_2.
\]
Since $\,\nu_i=0$ for $i\ge 2\,,\,$ using the properties of the function $Z_k$, the inequality takes the form
\[
(1\ +\ \pi_1\,Z_1^2)\,C_1^{-2}\,\pi_1\,\pi_1\,2\,Z_1^3\,Z_2\ >\ C_1^{-2}\,\pi_1\,Z_1^2\,\left(
\pi_2Z_2^2\,+\,\pi_1\,\pi_1\,2\,Z_1^3\,Z_2\,+\,(\pi_1Z_1)^2\,Z_2^2
\right).
\]
Opening the brackets,
\[
\pi_1\,\pi_1\,2\,Z_1^3\,Z_2\ >\ \pi_1\,Z_1^2\,\left(
\pi_2Z_2^2\,+\,(\pi_1Z_1)^2\,Z_2^2\right).
\]
Consider the case $C_1=C_2$ in which $Z_1=Z_2$. Then, the inequality takes the form
\[
2\,\pi_1\ >\ \pi_2\ +\ (\pi_1\,Z_1)^2.
\]
If $\,\pi_2\ >\ 2\,\pi_1,$ this cannot hold.

\section{Proofs for Section \ref{optimalitysection}, Optimality}

\begin{proof}[Proof of Lemma \ref{verification}]

Let $\phi$ be any search-effort process, and let
 \begin{eqnarray*}
 \theta^\phi_t&=&-\int_0^te^{-rs}K(\phi_{s})\, ds+ e^{-rt}V(N^\phi_t), \quad t<\tau\\
              &=&-\int_0^{\tau}e^{-rs}K(\phi_{s})\, ds+ e^{-r\tau}u(N^\phi_\tau), \quad t\geq \tau.
              \end{eqnarray*}
By It\^o's Formula and the fact that $V$ solves the HJB equation
(\ref{Bellman}), $\theta^\phi$ is a super-martingale, so
 \begin{equation}
 \label{opt1}
 V(N_{i0})=\theta^\phi_0\geq U(\phi).
 \end{equation}

For the special case of $\phi^*_t=\Gamma(N_t)$, where $N$ satisfies the stochastic
differential equation associated with the specified search-effort policy function $\Gamma$,
the HJB equation implies that  $\theta^*=\theta^{\phi^*}$ is actually a martingale.
Thus,
\[V(N_{i0})=\theta^*_0=  E(\theta^*_\tau)=U(\phi^*).\]
It follows from (\ref{opt1}) that $U(\phi^*)\geq U(\phi)$ for any control $\phi$.
\end{proof}

\begin{lemma}\label{Moper} The operator $\,M\,:\,\ell_\infty\to \ell_\infty\,,\,$ defined by
\begin{multline}
(MV)_n\ =\ \max_{c}\,\Bigg(
\frac{\eta'\,u(n)\ -\ K(c)}{c\bar C+r+\eta'}\ +\ \frac{c}{c\bar C+r+\eta'}\,\sum_{m=1}^\infty\,V_{n+m}\,C_m\,\mu^{C}_m\ \Bigg),
\end{multline}
is a contraction, satisfying, for candidate value functions $V_1$ and $V_2$,
\[
\|MV_1\ -\ MV_2\|_\infty\ \le\ \frac{c_H\bar C}{c_H\bar C+r+\eta'}.
\]
In addition, $\,M\,$ is monotonicity preserving.
\end{lemma}

\begin{proof} The fact that $\,M\,$ preserves monotonicity follows because
the pointwise maximum of two monotone increasing functions is also monotone. Now we provide the contraction argument. Let
\begin{equation}
M_{c}V_n\ =\ \frac{\eta'\,u(n)\ -\ K(c)}{c\bar C+r+\eta'}\ +\ \frac{c}{c\bar C+r+\eta'}\,\sum_{m=1}^\infty\,V_{n+m}\,C_m\,\mu^{C}_m.\
\end{equation}
Let
\[
MV_1(n)\ =\ M_{c_1}V_1(n)
\]
and
\[
MV_2(n)\ =\ M_{c_2}V_2(n),
\]
and assume without loss of generality that $\,MV_1(n)\ \le\ MV_2(n)\,.$ Then,
\begin{eqnarray*}
0\ &\le& \ MV_2(n)\ -\ MV_1(n)\ \\
&=&\ M_{c_2}V_2(n) -\ M_{c_1}V_1(n)\ \\
&\le & \ M_{c_2}V_2(n) -\ M_{c_2}V_1(n)\ \\
&=& \ \frac{c_2}{c_2\bar C+r+\eta'}\,\sum_{m=1}^\infty\,(V_2-V_1)_{n+m}\,C_m\,\mu^{C}_m,\
\end{eqnarray*}
and the estimate immediately follows.
\end{proof}

The contraction mapping theorem implies the

\begin{cor}\label{fixedp} The value function $V$ is the unique fixed point of $\,M\,$ and is a monotone increasing function.
\end{cor}

\begin{lemma}\label{udec} Let
\[
{\mathcal L}\ =\ \{V\,\in\,\ell_\infty\,:\,V_n\,-\,(r+\eta')^{-1}\,\eta'\,u_n\ \text{is monotone decreasing}\}.
\]
If $\,u(\,\cdot\,)\,$ is concave, then the operator $\,M\,$ maps $\,\mathcal L\,$ into itself.
 Consequently, the unique fixed point of $\,M\,$ also belongs to $\,\mathcal L\,.$
\end{lemma}

\begin{proof} Suppose that $\,V\,\in\,{\mathcal L}\,.$ Using the identity
\[
\frac{u(n)}{c\bar C+r+\eta'}\ -\ \frac{u(n)}{r+\eta'}\ =\ -\,\frac{u(n)}{r+\eta'}\, \frac{c\bar C}{c\bar C+r+\eta'},
\]
we get
\begin{eqnarray*}
&& \!\!\!\!\!\!\!\!\!\!\!\!\!\!\!\!\!\! M_{c}V_n\ -\ \frac{\eta'\,u_n}{r+\eta'}\ \\
&=& \ \frac{\eta'\,u_n\ -\ K(c)}{c\bar C+r+\eta'}\ +\ \frac{c}{c\bar C+r+\eta'}\,\sum_{m=1}^\infty\,V_{n+m}\,C_m\,\mu^{C}_m\\
& = & \ \frac{-K(c)}{c\bar C+r+\eta'}\
+\ \frac{c}{c\bar C+r+\eta'}\,\sum_{m=1}^\infty\,\left(V_{n+m}\,-\,\frac{\eta'\,u_{n+m}}{r+\eta'}\,+\,\frac{\eta'\,u_{n+m}
\,-\,\eta'\,u_n}{r+\eta'}\right)\, C_m\,\mu^{C}_m.
\end{eqnarray*}
Since $\,V\,\in\,{\mathcal L}\,$ and $\,u\,$ is concave, the sequences
\begin{eqnarray}\label{Dn}
D_n\ & =& \ \sum_{m=1}^\infty\left(V_{n+m}\,-\,\frac{\eta'\,u_{n+m}}{r+\eta'}\,+\,\frac{\eta'\,u_{n+m}\,-\,\eta'\,u_n}{r+\eta'}\right)\,C_m\,\mu^{C}_m\\
& =& \ \sum_{m=1}^\infty V_{n+m}\,C_m\,\mu^{C}_m\ -\ \frac{\eta'\,u_n}{r+\eta'}
\end{eqnarray}
and
\begin{eqnarray}
\label{Bn}
B_n\ & =& \ \sum_{m=1}^\infty\,\left(V_{n+m}\,-\,\frac{\eta'\,u_{n+m}}{r+\eta'}\,+\,\frac{\eta'\,u_{n+m}\,-\,\eta'\,u_n}{r+\eta'}\right)\,
C_m\,\mu^{C}_m\ \\
&=& \ \sum_{m=1}^\infty\,V_{n+m}\,C_m\,\mu^{C}_m\ -\ \frac{\bar C\,\eta'\,u_n}{r+\eta'}
\end{eqnarray}
are monotone decreasing.
Since the maximum of  decreasing sequences is again decreasing, the sequence
\[
(MV)_n\ -\ \frac{\eta'\,u_n}{r+\eta'}\ =\ \max_{c}\,(M_{c}V)_n\ -\ \frac{\eta'\,u_n}{r+\eta'}\
\]
is decreasing, proving the result.
\end{proof}

We will need the following auxiliary

\begin{lemma}\label{K(c)} If $K(c)$ is a convex differentiable function, then $c\mapsto \,K(c)\ -\ K'(c)\,c\,$ is monotone decreasing for $c\,>\,0\,.$
\end{lemma}

\begin{proof} For any $c$ and $b$ in $[c_L,c_H]$ with $c\geq b$, using first the convexity property that $K(c)-K(b)\leq K'(c)(c-b)$ and
then the fact that the derivative of a convex function is an increasing function, we have
\begin{eqnarray*}
(K(c)\ -\ K'(c)\,c) -(K(b)-K'(b)b)\ &\leq & \ K'(c)(c-b) -K'(c)c+K'(b)b\\
                                                  &=& b(K'(b)-K'(c)) \le\ 0,
\end{eqnarray*}
the desired result.
\end{proof}

\begin{prop} Suppose that the search cost function $K$ is convex, differentiable, and
increasing. Given $(\mu,C)$,  any optimal search-effort policy function $\Gamma( \,\cdot\,)$ is monotone decreasing.
If $K(c) = \kappa c$ for some $\kappa>0$, then there is an optimal policy of the
trigger form. \end{prop}

\bigskip

\begin{proof} The optimal $\,V\,$ solves the alternative Bellman equation
\begin{equation}
V_n\ =\ \max_{c}\,\Bigg(\frac{\eta'\,u_n\ -\ K(c)}{c\bar C+r+\eta'}\ +\ \frac{c}{c\bar C+r+\eta'}\,
\sum_{m=1}^\infty\,V_{n+m}\,C_m\,\mu^{C}_m\,\Bigg).\label{eq:bellman}
\end{equation}
We want to solve
\[
\max_{c}\,f(c)\ =\ \max_c\,\frac{\,\eta'u(n)+\,c\,Y_n\ -\ K(c)}{c\,\bar C+r+\eta'}
\]
with
\[
Y_n\ =\ \sum_{m=1}^\infty\,V_{n+m}\,C_m\,\mu^{C}_m\, .
\]
Then,
\[
f'(c)\ =\ \frac{(Y_n(r+\eta')\ -\,\eta'\,u(n)\,\bar C)\ +\ \bar C\,(K(c)\ -\ K'(c)\,c)\ -\ (r+\eta')\,K'(c)}{(c\,\bar C+r+\eta')^2}.
\]
By Lemma \ref{K(c)}, the function $\,K(c)\ -\ K'(c)\,c\,$ is monotone decreasing and the function $\,-\,(r+\eta')\,K'(c)\,$ is decreasing because $\,K(\,\cdot\,)\,$ is convex. Therefore, the function
\[
\bar C\,(K(c)\ -\ K'(c)\,c)\ -\ (r+\eta')\,K'(c)
\]
is also monotone decreasing.
There are three possibilities. If the unique solution $z_n$ to
\[
\bar C\,(K(z_n)\ -\ K'(z_n)\,z_n)\ -\ (r+\eta')\,K'(z_n)\ +\ (Y_n(r+\eta')\ -\,\eta'\,u(n)\,\bar C)\ =\ 0
\]
belongs to the interval $[c_L,c_H]$, then $f'(c)$ is positive for $c<z_n$ and is negative for $c>z_n\,.$
 Therefore, $f(c)$ attains its global maximum at $z_n$ and the optimum is $c_n=z_n\,.$ If
\[
\bar C\,(K(c)\ -\ K'(c)\,c)\ -\ (r+\eta')\,K'(c)\ +\ (Y_n(r+\eta')\ -\,\eta'\,u(n)\,\bar C)\ >\ 0
\]
for all $c\,\in\,[c_L,c_H]\,$ then $f'(c)>0,$ so $f$ is increasing and the optimum is $c\ =\ c_H\,.$ Finally, if
\[
\bar C\,(K(c)\ -\ K'(c)\,c)\ -\ (r+\eta')\,K'(c)\ +\ (Y_n(r+\eta')\ -\,\eta'\,u(n)\,\bar C)\ <\ 0
\]
for all $c\,\in\,[c_L,c_H]\,$ then $f'(c)<0,$ so $f$ is decreasing and the optimum is $c\ =\ c_L\,.$
By \eqref{Bn}, the sequence
\[
Y_n(r+\eta')\ -\ \bar C\,\eta'\,u(n)\ =\ (r+\eta')\,B_n
\]
is monotone decreasing. The above analysis directly implies that the optimal policy is then also decreasing.

If $K$ is linear, it follows from the above that the optimum is $c_H$ if $Y_n(r+\eta')\ -\ \bar C\,\eta'\,u(n)>0$ and $c_L$ if $Y_n(r+\eta')\ -\ \bar C\,\eta'\,u(n)<0\,.$ Thus, we have a trigger policy.
\end{proof}

\bigskip

\begin{proof}[Proof of Proposition \ref{flattail}]
Using the fact that
\[
(r+\eta')\,V_n\ =\ \sup_{c\,\in \, [c_L,c_H]\,}\, \eta'\,u_n\,-\,K(c)\ +\ c\,\sum_{m=1}^\infty\,(V_{n+m}-V_n)C_m\,\mu_m,
\]
which is a concave maximization problem over a convex set,
the supremum is achieved at $c_L$
if and only if some element of the supergradient of the objective function at $c_L$ includes
zero or a negative number.
(See Rockafellar (1970).) This
is the case provided that
\[
\sum_{m=1}^\infty\,(V_{n+m}-V_n)C_m\,\mu_m\ \le\ K'(c_L),
\]
where $K'(c_L)$.
By Lemma \ref{udec},
\begin{eqnarray*}
\sum_{m=1}^\infty\,(V_{n+m}-V_n)C_m\,\mu_m\
 & \le &\ \eta'\,(r+\eta')\,\sum_{m=1}^\infty\,(u_{n+m}-u_n)C_m\,\mu_m \\
&\le &\  \eta'(r+\eta')\,\sum_{m=1}^\infty\,(\overline u-u_n)C_m\,\mu_m \\
&<& \ K'(c_L), \quad\quad \quad {\rm for}\,\,  n\,>\,\overline N,
\end{eqnarray*}
 completing the proof.
\end{proof}

\section{Proofs for Section 6: Equilibrium}

This appendix contains proofs of the results in Section 6.

\subsection{Monotonicity of the Value Function in Other Agents'
  Trigger Level}

From the results of the Section \ref{optimalitysection}, we can restrict
attention to equilibria in the form of a trigger policy $C^N$, with trigger precision level $N$.
For any constant $c$ in $[c_L,c_H]$, we define the operator $L^N_c$, at any bounded increasing sequence $g$, by
\begin{equation*}
(L^N_cg)_n= \frac{1}{c_H^2+\eta'+r} \left\lbrace \eta' u_n - K(c) + (c_H^2 - c \bar C^N)
g_n + c \sum_{m=1}^{\infty} g_{n+m} C^N_m \mu^N_m \right\rbrace,
\end{equation*}
where
\begin{equation*}
\bar C^N = \sum_{i=1}^{\infty} C^N_i \mu^N_i.
\end{equation*}

\begin{lemma} \label{Lsup} Given $(\mu^N,C^N)$, the  value function $V^N$ of any given
agent solves
\begin{equation*}
V^N_n=\sup_{c\in \lbrace c_L\,,\,c_H \rbrace} \left\lbrace L^N_c V^N_n \right\rbrace.
\end{equation*}
\end{lemma}

\begin{lemma}\label{mon} The operator $\,L^N_c\,$ is monotone
increasing in $N$. That is, for any increasing sequence $\,g\,,\,$
\[
L^{N+1}_c\,g\ \ge\ L^N_c\,g.
\]
\end{lemma}

\begin{proof}
It is enough to show that if $f(n)$ and $g(n)$ are increasing, bounded
functions with $\,f(n)\, \geq\, g(n)\,\geq 0\,$, then $L^{N+1}_c f(n) \geq
L^N_c g(n)$. For that it suffices to show that
\begin{multline}
(c^2_H -c \bar C^{N+1}) f(n) + c \sum_{m=1}^{\infty} f(n+m) C^{N+1}_m
\mu^{N+1}_m \\
\!\!\!\!\!\!\!\!\!\!\!\!\!\!\!\!\!\!\!\!\!\!\!\!\!\!\!\!\!\!\!\!\!\!\!\!\!\!\!\!\!\!
\geq (c^2_H -c \bar C^N) g(n) + c \sum_{m=1}^{\infty} g(n+m) C^N_m \mu^N_m.
\label{ineqmon}
\end{multline}
Because $f$ and $g$ are increasing and $f\geq g$, inequality
\eqref{ineqmon} holds because
\begin{equation*}
\sum_{m=k}^{\infty} C_m^{N+1} \mu_m^{N+1} \geq \sum_{m=k}^{\infty}
C_m^N \mu_m^N
\end{equation*}
for all $k \geq 1$, based on Proposition \ref{M>N}.
\end{proof}

\begin{prop}\label{monVN}
If $N' \geq N$, then $V^{N'}(n) \geq V^N(n)$ for all $n$.
\end{prop}

\begin{proof} Let
\[
L^NV\ =\ \sup_{c\,\in\,[c_L\,,\,c_H]}\,L^N_c\,V.
\]
By Lemmas \ref{Lsup} and \ref{mon},
\[
V^{N'}\ =\ L^{N'}\,V^{N'}\ \ge\ L^{N}\,V^{N'}.
\]
Thus, for any $\,c\,\in\,[c_L\,,\,c_H]\,,\,$
\begin{eqnarray*}
V^{N'}(n) & \ge & L^{N}_cV^{N'}(n) \\
         & = & \frac{1}{c_H^2+\eta'+r} \left[
          \eta' u(n) - K(c) + (c_H^2 - c \bar C^N)
          V^{N'}(n) + c \sum_{m=1}^{\infty} V^{N'}(n+m) c^N_m \mu^N_m
          \right].
\end{eqnarray*}
Multiplying this inequality by $\,c_H^2+\eta+r\,,\,$ adding $\,(\bar C^N\,c-\,c_H^2)\,V^{N'}(n)\,$, dividing by $\,c\bar C^N\,+\,r\,+\,\eta\,$, and
 maximizing over $\,c\,,\,$ we get
\[
V^{N'}\ \ge\ M\,V^{N'},
\]
where the $\,M\,$ operator is defined in Lemma \ref{Moper}, corresponding to $\,c^N\,.$ Since $\,M\,$ is monotone, we get
\[
V^{N'}\,\ge\,M^k\,V^{N'}
\]
for any $\,k\,\in\,\N\,.$ By Lemma \ref{Moper} and Corollary \ref{fixedp},
\[
\lim_{k\to\infty}\,M^k\,V^{N'}\ =\ V^N,
\]
and the proof is complete.
\end{proof}

\subsection{Nash Equilibria}

We define the operator ${\cal Q}: \ell_{\infty} \to \ell_{\infty}$ by
\begin{equation*}
({\cal Q} \tilde C)(n) = \argmax\limits_c
\Bigg(\frac{\eta'u(n)\ -\ K(c)}{\bar C c+r+\eta'}\ +\ \frac{c}{
\bar C c+r+\eta'}\,\sum_{m=1}^\infty\,V^{\tilde C}(n+m)\,\mu^{\tilde C}(m) \Bigg),
\end{equation*}
where
\begin{equation*}
V^{\tilde C}(n)\ =\ \max_{c}\,\Bigg(
\frac{\eta'u(n)\ -\ K(c)}{\bar Cc+r+\eta'}\ +\
\frac{c}{\bar Cc+r+\eta'}\,\sum_{m=1}^\infty\,V^{\tilde C}(n+m)\,
\mu^{\tilde C}(m)\ \Bigg).
\end{equation*}
We then define the  ${\cal N}(N) \subset {\mathbb N}$ as
\begin{equation*}
{\cal N}(N) = \left\lbrace n \in {\mathbb N}; C^{n} \in {\cal Q}(C^N)
\right\rbrace .
\end{equation*}

The following proposition is a direct consequence of Proposition
\ref{trigger_policy} and the definition of the
correspondence ${\cal N}$.
\begin{prop}
Symmetric pure strategy Nash equilibria of the game
are given by trigger policies with trigger precision
levels that are the fixed points of the correspondence ${\cal N}$.
\end{prop}

\begin{lemma} The correspondence ${\cal N}(N)$ is monotone increasing
in $N$.
\end{lemma}

\begin{proof} From the Hamilton-Jacobi-Bellman equation \eqref{Bellman},
\[
(r+\eta')\,V^N_n\ =\ \max_{c\in [c_L,c_H]}\, \eta' u_n\ +\,c\,\left(-\,\kappa \,+\,\sum_{m=1}^\infty\,\left(V^N_{n+m}
\,-\,V^N_n\right)\,C^N_m\,\mu_m^N\right).
\]
Let first $\,c_L\,>\,0\,.$ Then,  it is optimal for an agent to choose $c=c_H\,$ if
\begin{equation}\label{kappa}
-\,\kappa\,+\,\sum_{m=1}^\infty\,\left(V^N_{n+m}\,-\,V^N_n\right)\,C^N_m\,\mu_m^N\
>\ 0\ \Leftrightarrow (r+\eta')\,V^N_n\ -\ \eta' u_n\ >\ 0,
\end{equation}
the agent is indifferent between choosing $c_H$ or $c_L$ if
\begin{equation}\label{kappa1}
-\,\kappa\,+\,\sum_{m=1}^\infty\,\left(V^N_{n+m}\,-\,V^N_n\right)\,C^N_m\,\mu_m^N\
=\ 0\ \Leftrightarrow (r+\eta')\,V^N_n\ -\ \eta' u_n\ =\ 0,
\end{equation}
and the agent will choose $c_L$ if the inequality $<$ holds in
\eqref{kappa}. By Lemma \ref{udec}, the set of $n$ for which
$(r+\eta')\,V^N_n\ -\ \eta' u_n=0$ is either empty or is  an interval
$N_1\le n\le N_2.$ Proposition \ref{monVN} implies the required
monotonicity.

Let now $\,c_L\,=\,0.$ Then, by the same reasoning,  it is optimal for
an agent to choose $c=c_H\,$ if and only if $(r+\eta')\,V^N_n\ -\
\eta' u_n>0$. Alternatively, $(r+\eta')\,V^N_n\ -\ \eta' u_n=0$.
By Lemma \ref{udec}, the set $n$ for which $(r+\eta')\,V^N_n\ -\ \eta'
u_n>0$ is an interval $n\ <\ N_1$ and hence $(r+\eta')\,V^N_n\ -\
\eta' u_n=0$ for all $n\ge N_1.$ Consequently, since $u_n$ is monotone
increasing and concave, the sequence
\[
Z_n\ :=\ -\,\kappa\,+\,\sum_{m=1}^\infty\,\left(V^N_{n+m}\,-\,V^N_n\right)\,C^N_m\,\mu_m^N\ =\ -\,\kappa\,+\,\frac{\eta'}{r+\eta'}\sum_{m=1}^\infty\,\left(u_{n+m}\,-\,u_n\right)\,C^N_m\,\mu_m^N
\]
is decreasing for $n\ge N_1.$ Therefore, the set of $n$ for which
$Z_n =0$ is either empty or is an interval $N_1 \leq n \leq N_2$.
Proposition \ref{M>N} implies the required monotonicity.
\end{proof}

\begin{prop}\label{fixN} The correspondence ${\cal N}$ has at least one fixed
point. Any fixed point is less than or equal to $\overline N\,.$
\end{prop}

\begin{proof} If $0 \in \,{\cal N}(0),$ we are done. Otherwise,
$\inf\lbrace {\cal N} (0) \rbrace \ge 1\,.$ By monotonicity,
$\inf \lbrace \,{\cal N}(1)\ \rbrace \ge\ 1\,.$ Again, if
$1 \in \,{\cal N}(1),$ we are done. Otherwise, continue inductively. Since
there is only a finite number $\,\overline N\,$ of possible
outcomes, we must arrive at some $n$ in ${\cal N}(n)$.
\end{proof}

\bigskip

\begin{proof}[Proof of Proposition \ref{pareto}]
The result follows directly from Proposition \ref{monVN}
and the definition of equilibrium.
\end{proof}

\subsection{Equilibria with minimal search intensity}

\begin{lemma}\label{mon-pres}
The operator $\,A\,$ is a contraction on $\,\ell_\infty\,$ with
\[
\|A\|_{\ell_\infty\to \ell_\infty}\ \le\ \frac{c_L^2}{r+\eta'+c_L^2}\ <\ 1.
\]
Furthermore, $\,A\,$ preserves positivity, monotonicity, and concavity.
\end{lemma}

\begin{proof}[Proof of Lemma \ref{I-A}] We have, for any bounded sequence $g$,
\begin{equation}
|A(g)_n|\ \le\ \,\frac{c_L^2}{r\,+\,\eta'\,+\,c_L^2}\,\sum_{m=1}^\infty\,|g_{n+m}|\,\mu_m\
\le\ \frac{c_L^2}{r+\eta'+c_L^2}\,\sup_n\,g_n,
\end{equation}
which establishes the first claim.
Furthermore, if $\,g\,$ is increasing, we have $\,g_{n_1+m}\ \ge\ g_{n_2+m}\,$ for $\,n_1\ge n_2\,$, and for any $\,m\,.$
Summing up these inequalities, we get the required monotonicity. Preservation of concavity is proved similarly.
\end{proof}

\bigskip

\begin{proof}[Proof of Theorem \ref{homog}] It suffices to show that, taking the minimal-search
$(\mu^0,C^0)$ behavior as given, the minimal-effort search effort $c_L$
achieves the supremum defined by the Hamilton-Jacobi-Bellman equation, at each precision $n$, if and only if
$K'(c_L)\geq B$, where $B$ is defined by \eqref{foc}.

By the previous lemma, $\,V(n)\,$ is concave in $\,n\,$. Thus, the sequence
\[
\sum_{m=1}^\infty\,(V_{n+m}\,-\,V_n)\,\mu^0_m
\]
is monotone decreasing with $\,n\,.$ Therefore,
\[
\sum_{m=1}^\infty\,(V_{1+m}\,-\,V_1)\,\mu^0_m\ =\ \max_n\,\sum_{m=1}^\infty\,(V_{n+m}\,-\,V_n)\,\mu^0_m.
\]
We need to show that the objective function, mapping $c$ to
\[
-K(c)\ +\ c\,c_L\,\sum_{m=1}^\infty\,(V_{n+m}\,-\,V_n)\,\mu^0_m,
\]
achieves its maximum at $\,c\,=\,c_L\,.$ Because $\,K$ is convex, this
objective function is decreasing on $\,[c_L\,,\,c_H]\,,\,$ and the claim follows.
\end{proof}

The following lemma follows directly from the proof of Proposition \ref{fixN}.

\begin{lemma} If $C^N\ \in\ {\cal Q}(C^1)\,$ for some $N\ge 1\,$ then the correspondence ${\cal N}\,$ has  a fixed point $n\ge 1\,.$
\end{lemma}

\begin{proof}[Proof of Proposition \ref{more0}] By the above lemma, it suffices to show that ${\cal Q}(C^1)\,\not=\,C^0\,.$ Suppose on the contrary that ${\cal Q}(C^1)\ =\ C^0\,.$ Then the value function is simply
\[
V_n^1\ =\ \frac{\eta'\,u(n)}{r+\eta'}.
\]
It follows from \eqref{kappa} that $C_1\ =\ 0$ is optimal if and only if
\[
\frac{\eta'\,(u(2)\,-\,u(1))\,c_H\mu_1^1}{r+\eta'}\ =\ \sum_{m=1}^\infty\,\left(V^1_{n+m}\,-\,V^N_n\right)\,C^1_m\,\mu_m^N\
<\ \kappa.
\]
By \eqref{kappa}, we will have the inequality $\,V^1_n\,\ge\,V^0_n$ for all $n\,\ge\,2$ and a strict inequality $V^1_1\ >\ V^0_1,$ which is a contradiction.
The proof is complete.
\end{proof}

\section{Proofs of Results in Section \ref{interventions}, Policy Interventions}

\begin{proof}[Proof of Lemma \ref{mon_subsidy}] By construction, the value function $V^\gd$
associated with proportional
search subsidy rate $\delta$ is monotone increasing in $\delta\,.$ By \eqref{kappa}, the optimal trigger $N$ is that
$n$ at which the sequence
\[
(r+\eta')\,V^\gd_n\ -\ \eta'\,u(n)
\]
crosses zero. Hence, the optimal trigger is also monotone increasing in $\gd\,.$
\end{proof}

\bigskip

\begin{proof}[Proof of Lemma \ref{mon_educ}] Letting $V_{n,M}$ denote the value associated
with $n$ private signals and $M$ public signals, we define
\[
Z_{n,M}\ =\ V_{n,M}\ -\ \frac{\eta' u_{n+M}}{r\,+\,\eta'}.
\]
We can rewrite the HJB equation for the agent, educated with $M$ signals at birth, in the form
\begin{eqnarray*}
(r+\eta')\,Z_{n,M}\ &=&\ \sup_{c\in[c_L,c_H]}\,c\left(\,-\kappa\,+\,\sum_{m=1}^\infty\,(V_{n+m,M}-V_{n,M})\,\mu_m^C\,\right)\ \\
&& \\
&=&\  \sup_{c\in[c_L,c_H]}\,c\left(\,W_{n,M}\ -\ \kappa\,+\,\sum_{m=1}^\infty\,(Z_{n+m,M}-Z_{n,M})\,\mu_m^C\,\right),
\end{eqnarray*}
where
\[
W_{n,M}\ =\ \frac{\eta'}{r+\eta'}\,\sum_{m=1}^\infty\,(u_{n+M+m}-u_{n+M})\,\mu_m^C.
\]
The quantity $W_{n,M}$ is monotone decreasing with $M$ and $n$ by the concavity of $u(\,\cdot\,).$
Equivalently,
\begin{equation}\label{supZ}
Z_{n,M}\ =\ \sup_{c\in[c_L,c_H]}\quad \frac{c}{\bar C\,c\,+\,r\,+\,\eta'}\,\left(\,W_{n,M}\ -\ \kappa\,+\,\sum_{m=1}^\infty\,Z_{n+m,M}\,\mu_m^C\,\right).\
\end{equation}
Treating the right-hand side
as the image of an operator at $Z$, this operator is a contraction and, by Lemma \ref{fixedp}, $Z_{n,M}$ is its unique fixed point. Since $W_{n,M}$ decreases with $M,$ so does $Z_{n,M}\,.$

By Lemma \ref{udec}, $Z_{n,M}$ is also monotone decreasing with $n\,.$ Hence, an optimal trigger policy, attaining the supremum in \eqref{supZ}, is an $n$ at which the sequence
\[
W_{n,M}\ -\ \kappa\,+\,\sum_{m=1}^\infty\,Z_{n+m,M}\,\mu_m^C
\]
crosses zero. Because both $W_{n,M}$ and $Z_{n,M}$ are decreasing with $M,$ the trigger is also decreasing with $M\,.$
\end{proof}

\bigskip

\begin{proof} [Proof of Proposition \ref{equil_educ}] Suppose that $C^N$ is an equilibrium with  $M$ public signals,
which we express as
\[
C^N\,\in\,{\cal Q}_M(C^N).
\]
By Lemma \ref{mon_educ}, we have
\[
C^{N_1}\ \in\ {\cal Q}_{M-1}(C^N)
\]
for some $N_1\ge\ N\,.$ It follows from the algorithm at the end of Section \ref{sec_equil} that there exists some
 $N'\,\ge\,N$ with $N'\in\,{\cal Q}_{M-1}(C^{N'}).$ The proof is completed by induction in $M$.
\end{proof}

\bigskip
\newpage

\bigskip\noindent
\centerline{\large \bf References}

\bigskip\noindent
Amador, Manuel and Weill, Pierre-Olivier. 2007. ``Learning from
Private and Public Observations of Others' Actions.'' Working
Paper, Stanford University.

\bigskip\noindent
Arrow, Kenneth. 1974. {\it The Limits of Organization}. Norton, New York.

\bigskip\noindent
Banerjee, Abhijit. 1992. ``A Simple Model of Herd Behavior.''
{\it Quarterly Journal of Economics}, {\bf 107:} 797-817.

\bigskip\noindent
Bikhchandani, Sushil, David Hirshleifer, and Ivo Welch. 1992.
``A Theory of Fads, Fashion, Custom, and Cultural Change as
Informational Cascades.''
{\it Journal of Political Economy}, {\bf 100:} 992-1026.

\bigskip\noindent
Banerjee, Abhijit, and Drew Fudenberg. 2004. ``Word-of-Mouth Learning.''
{\it Games and Economic Behavior}, {\bf 46:} 1-22.

\bigskip\noindent
Billingsley, Patrick. 1986. {\it Probability and Measure, Second
Edition}, Wiley, New York.

\bigskip\noindent
Blouin, Max and Roberto Serrano. 2001. ``A Decentralized Market with
Common Value Uncertainty: Non-Steady States.''
{\it Review of Economic Studies}, {\bf 68:} 323-346.

\bigskip\noindent
Dieudonn\'e, Jean  1960. {\it Foundations of Modern Analysis.}
Academic press, New York-London.

\bigskip\noindent
Duffie, Darrell, Nicolae G\^arleanu, and Lasse Pedersen. 2005.
``Over-the-Counter Markets.'' {\it Econometrica}, {\bf 73:}
1815-1847.

\bigskip\noindent
Duffie, Darrell, Gaston Giroux, and Gustavo Manso. 2008.
``Information Percolation.'' Working Paper, MIT.

\bigskip\noindent
Duffie, Darrell, and Gustavo Manso. 2007. ``Information Percolation in
Large Markets.''
{\it American Economic Review Papers and Proceedings}, {\bf 97:} 203-209.

\bigskip\noindent
Duffie, Darrell,  and Yeneng Sun. 2007. ``Existence of Independent
Random Matching.'' {\it Annals of Applied Probability}, {\bf 17:}
386-419.


\bigskip\noindent
Gale, Douglas. 1987. ``Limit Theorems for Markets with Sequential
Bargaining.'' {\it Journal of Economic Theory}, {\bf 43:}
20-54.

\bigskip\noindent
Grossman, Sanford.  1981. ``An Introduction to the Theory of Rational
Expectations Under Asymmetric Information.'' {\it Review of Economic
Studies}, {\bf 4:} 541-559.

\bigskip\noindent
Hartman, Philip. 1982. {\it Ordinary Differential Equations, Second
Edition}, Boston: Birkha\"user.

\bigskip\noindent
Hayek, Friedrich. 1945. ``The Use of Knowledge in Society.'' {\it American Economic Review}, {\bf 35:} 519-530.

\bigskip\noindent
Kiyotaki, Nobuhiro, and Randall Wright. 1993. ``A Search-Theoretic
Approach to Monetary Economics.''
{\it American Economic Review}, {\bf 83:}
63-77.

\bigskip\noindent
Milgrom, Paul. 1981. ``Rational Expectations, Information Acquisition, and Competitive Bidding.''
{\it Econometrica}, {\bf 50:} 1089-1122.

\bigskip\noindent
Pesendorfer, Wolfgang and Jeroen Swinkels. 1997. ``The Loser's Curse and Information Aggregation in
Common Value Auctions.'' {\it Econometrica}, {\bf 65:} 1247-1281.

\bigskip\noindent
Pissarides, Christopher. 1985. ``Short-Run Equilibrium Dynamics
of Unemployment Vacancies, and Real Wages.'' {\it American Economic
Review}, {\bf 75:} 676-690.

\bigskip\noindent
Protter, Philip. 2005.  {\it Stochastic Differential and Integral Equations, Second Edition}, New York: Springer.

\bigskip\noindent
Reny, Philip and Motty Perry. 2006. ``Toward a Strategic Foundation for Rational Expectations Equilibrium.''
{\it Econometrica}, {\bf 74:} 1231-1269.

\bigskip\noindent
R. Tyrrell Rockafellar. 1970. {\it Convex Analysis}, Princeton University Press.

\bigskip\noindent
Rubinstein, Ariel and Asher Wolinsky. 1985. ``Equilibrium in a Market
with Sequential Bargaining.'' {\it Econometrica},  {\bf 53:}
1133-1150.

\bigskip\noindent
Sun, Yeneng. 2006. ``The Exact Law of Large Numbers via Fubini Extension and
Characterization of Insurable Risks.'' {\it Journal of Economic Theory}, {\bf 126:} 31-69.

\bigskip\noindent
Trejos, Alberto, and Randall Wright. 1995. ``Search, Bargaining,
Money, and Prices.'' {\it Journal of Political Economy}, {\bf 103:},
118-141.

\bigskip\noindent
Vives, Xavier. 1993. ``How Fast do Rational Agents Learn.''
{\it Review of Economic Studies}, {\bf 60:}
329-347.


\bigskip\noindent
Wilson, Robert. 1977. ``Incentive Efficiency of Double Auctions.''
{\it The Review of Economic Studies}, {\bf 44:} 511-518.

\bigskip\noindent
Wolinsky, Asher. 1990. ``Information Revelation in a Market With
Pairwise Meetings.'' {\it Econometrica}, {\bf 58:} 1-23.

\bigskip\noindent
Yeh, J. 2006. {\it Real analysis. Theory of measure and integration, Second Edition.} Singapore: World Scientific Publishing.

\end{document}